\documentclass{birkmult}
\usepackage[ukrainian]{babel}
\usepackage{afterpage,float,amsmath}
\usepackage{longtable}

 \newtheorem{thm}{Theorem}[section]

 \newtheorem{prop}[thm]{Proposition}
 \theoremstyle{definition}
 \newtheorem{defn}[thm]{Definition}
 \theoremstyle{remark}
 \newtheorem{rem}[thm]{Remark}
 
 \numberwithin{equation}{section}

\begin{document}
\renewcommand{\abstractname}{Abstract}
\renewcommand{\proofname}{Proof}
\renewcommand{\refname}{References}
%
%
%
%
%
%
%
%
%
\title[On the unitarization of linear representations of posets]
 {On the unitarization of linear representations of primitive partially ordered sets }
\author{Roman Grushevoy}

\address{%
Institute of Mathematics  NAS of Ukraine\\
Tereschenkivska str. 3\\
01601 Kiev\\
Ukraine}

\email{grushevoy@imath.kiev.ua}

\thanks{This work has been partially supported by the Scientific Program of National Academy
of Sciences of Ukraine, Project No~0107U002333. }
\author{Kostyantyn Yusenko}
\address{%
Institute of Mathematics  NAS of Ukraine\\
Tereschenkivska str. 3\\
01601 Kiev\\
Ukraine}

\email{kay@imath.kiev.ua}


\date{April 11, 2008}
\begin{abstract}
We describe all weights which are appropriated for the unitarization
of linear representations of primitive partially ordered sets of
finite type.
\end{abstract}
\maketitle

\section*{0. Introduction}

The representation theory of partially ordered sets (posets) in
linear vector spaces has became one of the classical field in
linear algebra (see~\cite{NazarovaRoiter, Kleiner1, Kleiner2, Drozd} and other). 
See Section~1 for the details.

On other hand when one try to develop the similar theory for
Hilbert spaces, one will be faced with difficulties even with
three subspace two of which is orthogonal: it is impossible to
obtain their description in reasonable way (it is so called
$*$-wild problem) see~\cite{KruglyaSamoilenko1,KruglyaSamoilenko2}
. But adding linear relation
$$
    \alpha_1P_1+\ldots+\alpha_nP_n=\gamma I,
$$
between the projections $P_i$ on corresponding subspaces, in some
cases gives nice answers which have deep interconnections with
linear case (see Section~2 for the details).

The aim of this article is to show that each linear representation
of primitive poset of finite type could be unitarized with some
weight $\chi=(\alpha_1,\ldots,\alpha_n,\gamma)$ and to describe
all possible weights $\chi$ appropriated to the unitarization for
given linear representation of primitive poset (Section~3).

\section{Posets and their linear representations}

In this section we will briefly recall some results concerning to
partially ordered sets and their linear representation.


\subsection {Posets and Hasse diagrams}

Let $(\mathcal N,\prec)$ be finite partially ordered set (or poset
for short) which for us will be $\{1,\ldots,n\}$.
By the width of the poset $\mathcal N$ we will understand the
maximal number of two by two incomparable elements of this set.

The poset $\mathcal N$ of the width $m$ called primitive and denoted
by $(k_1,\ldots,k_m)$ if this set is the cardinal sum of $m$
linearly ordered sets $\mathcal N_1,\ldots,\mathcal N_m$ with orders
$k_1,\ldots,k_m$.

We will use the standard graphic representations for the poset
$(\mathcal N,\prec)$ called  Hasse diagram. This representation
associate to each point $x \in \mathcal N$ the vertex $v_x$ and the
arrow $v_x\rightarrow v_y$ if $x\prec y$ and there is no $z \in
\mathcal N$, such that $x\prec z \prec y$. For example let $\mathcal
N=\{1,2,3,4,5,6\}$ with the following order:
$$
    2\prec1, \quad 2\prec3, \quad 4\prec3, \quad 6\prec5,
$$
then the corresponding Hasse diagram is the following:

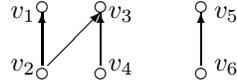
\begin{picture}(200,30)(-20,0)

\put(120,0){\vector(0,1){20}}
      \put(120,-3){\circle{4}}
      \put(120,22){\circle{4}}

\put(142,0){\vector(0,1){20}} \put(122,0){\vector(1,1){20}}
      \put(142,-3){\circle{4}}
      \put(142,22){\circle{4}}

\put(180,0){\vector(0,1){20}}
      \put(180,-3){\circle{4}}
      \put(180,22){\circle{4}}

\put(108,-2){$v_2$} \put(108,18){$v_1$} \put(145,-2){$v_4$}
\put(145,18){$v_3$} \put(185,-2){$v_6$} \put(185,18){$v_5$}

\end{picture}

\vspace{0.3cm}

\noindent For each primitive poset $\mathcal N=(k_1,...k_m)$ its
Hasse diagram has the following form:

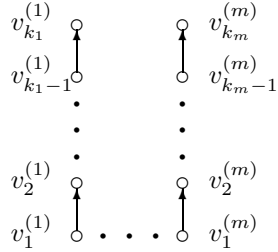
\begin{picture}(120,100)

\put(150,2){\vector(0,1){16}}
      \put(150,0){\circle{4}}
      \put(150,20){\circle{4}}
      \put(150,30) {\circle*{2}}
      \put(150,40) {\circle*{2}}
      \put(150,50) {\circle*{2}}

      \put(160,0) {\circle*{2}}
      \put(170,0) {\circle*{2}}
      \put(180,0) {\circle*{2}}

\put(150,62){\vector(0,1){16}}
      \put(150,60){\circle{4}}
      \put(150,80){\circle{4}}

\put(125,-2){$v^{(1)}_1$} \put(125,18){$v^{(1)}_2$}
\put(125,58){$v^{(1)}_{k_1-1}$} \put(125,78){$v^{(1)}_{k_1}$}

\put(190,2){\vector(0,1){16}}
      \put(190,0){\circle{4}}
      \put(190,20){\circle{4}}
      \put(190,30) {\circle*{2}}
      \put(190,40) {\circle*{2}}
      \put(190,50) {\circle*{2}}

\put(190,62){\vector(0,1){16}}
      \put(190,60){\circle{4}}
      \put(190,80){\circle{4}}

\put(200,-2){$v^{(m)}_1$} \put(200,18){$v^{(m)}_2$}
\put(200,58){$v^{(m)}_{k_m-1}$} \put(200,78){$v^{(m)}_{k_m}$}

\end{picture}

\vspace{0.5cm}

\subsection {Linear representations of posets. Indecomposability and Bricks}

\noindent By the linear representation $\pi$ of the poset
$\mathcal N$ in some complex vector space $V$ we will understand
such correspondence to each element $i \in \mathcal N$ the
subspace $V_i \subset V$, that $i \prec j$ implies $V_i \subseteq
V_j$. We will denote the representation $\pi$  in the following
way: $\pi=(V;V_1,\ldots,V_n)$, where $n$ is the cardinality of
$\mathcal N$, and for the primitive posets $(k_1,\ldots,k_m)$ we
will use the following notation
$\pi=(V;V^{(1)}_1,\ldots,V^{(1)}_{k_1};\ldots;$
$V^{(m)}_1,\ldots,V^{(m)}_{k_m}).$

By the dimension vector $d_\pi$ of the representation $\pi$ we will
understand the vector $d_\pi=(d_0;d_1,\ldots,d_n)$, where $d_0=\dim
V$, $d_i=\dim V_i$ (correspondingly for representations of primitive
posets the dimensional vector will be denoted by
$d_\pi=(d_0;d^{(1)}_1,\ldots,d^{(1)}_{k_1};\ldots;d^{(m)}_1,\ldots,d^{(m)}_{k_m})$).

In fact the linear representations of poset $\mathcal N$ forms the
additive category $\rm{Rep}(\mathcal N)$, where the set of
morphisms $\rm{Mor}(\pi_1,\pi_2)$ between two representations
$\pi_1=(V;V_1,\ldots,V_n)$ and $\pi_2=(W;W_1,\ldots,W_n)$ consists
of such linear maps $C:V \rightarrow W$, that $C(V_i)\subset W_i$.
Two representations $\pi_1$ and $\pi_2$ of poset $\mathcal N$ are
isomorphic (or equivalent) if there exists invertible morphism $C
\in \rm{Mor}(\pi_1,\pi_2)$, i.e. there exist invertible linear map
$C:V \rightarrow W$ such that $C(V_i)=W_i$.

 One can define the
direct sum $\pi=\pi_1 \oplus \pi_2$ of two objects $\pi_1, \pi_2
\in \mathcal N$ in the following way: $$\pi=(V \oplus W; V_1
\oplus W_1,\ldots,V_n \oplus W_n).$$ Using the notion of direct
sum it is natural to define \emph{indecomposable} representations
as the representations that are not isomorphic to the direct sum
of two non-zero representations, otherwise representations are
called \emph{decomposable}. It is easy to show that representation
$\pi$ is indecomposable iff there exist no non-trivial idempotents
in $\rm{End}(\pi)$. The representation $\pi$ is called
\emph{brick} if 
there is no non-trivial endomorphism of this representation. One
can show that \emph{brick} implies \emph{indecomposability}. But
there exist indecomposable representations of posets which are not
'brick': let us take
$$ A_\alpha= \left(
         \begin{array}{cc}
           1 & \alpha \\
           0 & 1 \\
         \end{array}
       \right), \quad \alpha \in \mathbb R \backslash \{0\},$$
the corresponding representation
$\pi_{\alpha}=(V;V_1,V_2,V_3,V_4)$ of the poset $\mathcal
N=(1,1,1,1)$
\begin{align*}
    V=\mathbb C^2 \oplus \mathbb C^2; \quad V_1=\mathbb
    C^2 \oplus 0, \quad V_2=0 \oplus \mathbb C^2, \\
    V_3=\left\{ (x,x)\in \mathbb C^4 | x \in \mathbb C^2\right \}, \quad
    V_4=\left\{ (x,A_\alpha x)\in \mathbb C^4 | x \in \mathbb C^2\right \},
\end{align*}
is indecomposable but is not brick.

We call the representation $\pi$ \emph{non-degenera\-ted} if the
components of dimension vector $d_\pi$ satisfy the following
conditions:
\begin{itemize}
  \item $d_i\not=0$;
  \item if $i\prec j$ then $d_i<d_j$;
  \item $d_i<d_0,$
\end{itemize}
otherwise the representation is called \emph{degenerated}. Note that
in works~\cite{Kleiner2, Drozd} authors by term
\emph{non-degenera\-ted} means just first and second from above
conditions.

\subsection {Posets and Quivers}

To study the representations of the primitive posets it is helpful
to use the well-developed representation theory of quivers (for
the representations of quivers see for
example~\cite{GabrielRoiter,AssemSimson}).

A quiver $Q=(Q_0,Q_1,s,t:Q_1\rightarrow Q_0)$ is given by a set of
vertices $Q_0$, which for us will be $\{v_0,\ldots,v_n\}$, and a
set $Q_1$ of arrows. An arrow $\rho \in Q_1$ starts at the vertex
$s(\rho)$ and terminates at $t(\rho)$. A representation $X$ of $Q$
is given by a vector space $X_{v_i}$ for each vertex $v_i \in Q_0$
and linear operator $X_{\rho}:X_{s(\rho)}\rightarrow X_{t(\rho)}$.
In similar way as for the posets one can define the set of
morphisms between two representations $X$ and $Y$,
\emph{indecomposable} and \emph{bricks} representations
(see~\cite{GabrielRoiter} for the details).

To each poset $\mathcal N$ one can associate the quiver $Q^{\mathcal
N}$ in the similar way as Hasse diagram was associated, the only
difference is that we add the extra vertex $v_0$ to diagram and
connect by arrows all vertices that corresponds to maximal elements
of $\mathcal N$ with the vertex $v_0$. More formally, $Q^{\mathcal
N}_0=\{v_0,v_x | x \in \mathcal N\}$ and $Q^{\mathcal N}_1$ consists
of:

\begin{itemize}
     \item   the arrows $v_x \rightarrow v_y$ if $x \prec y$ and there are no such $z$ that
$x\prec z\prec y$, and

     \item   the arrows $v_x \rightarrow v_0$ if there no such $z$ that $x\prec
  z$.

   \end{itemize}
For example for each primitive poset $\mathcal N=(k_1,...k_m)$ the
corresponding quiver $Q^{\mathcal N}$ has the following form:

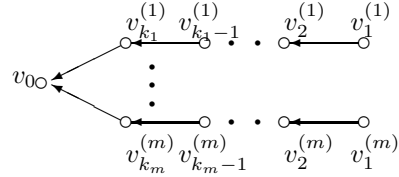
\begin{picture}(120,50)

\put(178,0){\vector(-1,0){26}}

\put(238,0){\vector(-1,0){26}}
      \put(150,0){\circle{4}}
      \put(180,0){\circle{4}}
      \put(190,0) {\circle*{2}}
      \put(200,0) {\circle*{2}}
      \put(210,0) {\circle{4}}
      \put(240,0) {\circle{4}}

\put(178,30){\vector(-1,0){26}}

\put(118,15){\circle{4}}

      \put(160,7) {\circle*{2}}
      \put(160,14) {\circle*{2}}
      \put(160,21) {\circle*{2}}

\put(148,29){\vector(-2,-1){26}} \put(148,1){\vector(-2,1){26}}

\put(238,30){\vector(-1,0){26}}
      \put(150,30){\circle{4}}
      \put(180,30){\circle{4}}
      \put(190,30) {\circle*{2}}
      \put(200,30) {\circle*{2}}
      \put(210,30) {\circle{4}}
      \put(240,30) {\circle{4}}

\put(107,15){$v_0$} \put(150,35){$v^{(1)}_{k_1}$}
\put(170,35){$v^{(1)}_{k_1-1}$} \put(210,35){$v^{(1)}_{2}$}
\put(235,35){$v^{(1)}_{1}$}

\put(150,-15){$v^{(m)}_{k_m}$} \put(170,-15){$v^{(m)}_{k_m-1}$}
\put(210,-15){$v^{(m)}_{2}$} \put(235,-15){$v^{(m)}_{1}$}

\end{picture}

\vspace{0.5cm}

\begin{rem} Notice that the underlying graph of the quiver
$Q_{\mathcal N}$ that corresponds to primitive poset $\mathcal
N=(k_1,\ldots,k_m)$ is the tree which denoted by
$T_{k_1+1,\ldots,k_m+1}$.
\end{rem}

It is obvious that each linear representations $\pi$ of the poset
$\mathcal N=(k_1,\dots,k_m)$ defines some representation $X_\pi$
of the corresponding quiver $Q^{\mathcal N}$. Indeed, let be given
the re\-pre\-sen\-tation
$$\pi=(V;V^{(1)}_1,\ldots,V^{(1)}_{k_1};\ldots;V^{(m)}_1,\ldots,V^{(m)}_{k_m})$$
of the poset $\mathcal N$ then the corresponding representation
$X_\pi$ of quiver $Q_{\mathcal N}$ could be built in the following
way: to each vertex $v^{(j)}_i$ we associate the space
$V^{(j)}_i$, $X_{v^{(j)}_i}=V^{(j)}_i$ (to the vertex $v_0$ we
associate the space $V$, $X_{v_0}=V$) and to each arrow $v^{(j)}_i
\rightarrow v^{(j)}_{i+1}$ we associate the embedding $V^{(j)}_i
\hookrightarrow V^{(j)}_{i+1}$ (to the arrows $v^{(j)}_{k_i}
\rightarrow v_0$ we associate the embedding $V^{(j)}_{k_i}
\hookrightarrow V$). And vice versa every representation $X$ of
quiver $Q$ with underlaying graph $T_{k_1+1,\ldots,k_m+1}$, such
that for all $\rho\in Q_1$ $X_{\rho}$ is monomorphism, defines a
representation of the poset $\mathcal N=(k_1,\dots,k_n)$:
$$\pi_X=(V;V^{(1)}_1,\ldots,V^{(1)}_{k_1};\ldots;V^{(n)}_1,\ldots,V^{(n)}_{k_n}),$$
where $V=X_{v_0}$ it is a representation of the the vertex $v_0$,
and $V^{(i)}_j=X_{v^{(i)}_j \rightarrow
v^{(i)}_{j+1}}\cdot\ldots\cdot X_{v^{(i)}_{k_i}\rightarrow
v_0}(X_{v^{(i)}_j})$ is an image of space corresponding to the
vertex $v^{(i)}_j$ under the maps that corresponds to the path
from vertex $v^{(i)}_j$ to vertex $v_0$.

In fact having take the representation of poset then using
construction above build the representation of corresponding quiver
and finally after building the representation of corresponding poset
one will get the representation which is isomorphic to that we start
on, i.e. $\pi\simeq\pi_{X_\pi}$ in the category $\rm{Rep}(\mathcal
N)$.

\begin{prop}
\begin{enumerate}
  \item The representations $\pi_1$ and $\pi_2$ of the poset $\mathcal
    N$ are isomorphic if and only if the corresponding representation
    $X_{\pi_1}$ and $X_{\pi_2}$ of the quiver $Q_{\mathcal N}$ are
    isomorphic.
  \item The representation $\pi$ is indecomposable if and only if
  the corresponding representation $X_\pi$ is indecomposable.
\end{enumerate}
\end{prop}

\begin{proof}

\begin{enumerate}

\item The morphism $C:X_{\pi_1}\rightarrow X_{\pi_2}$,
establish the isomorphism between the representations of the quiver
$Q_{\mathcal N}$ if and only if the linear map $C_{v_0}$ is a
isomorphism between the representations $\pi_1$ and $\pi_2$ of
$\mathcal N$.

\item If the representation  $\pi$ of primitive poset  $\mathcal
N$ is decomposable, i.e. $\pi\simeq\pi_1\oplus\pi_2$, then the
corresponding representation $X_\pi$ of the quiver $Q=Q_{\mathcal
N}$ has the following form: for all $v_i\in Q_0$
$X_{v_{i}}=U_i\oplus W_i$, for all $v_i\rightarrow v_j\in Q_1$
$X_{v_i\rightarrow
v_j}=\hat{\Gamma}_{i,j}\oplus\tilde{\Gamma}_{i,j}$, where
$\hat{\Gamma}_{i,j}:U_j\rightarrow U_i$ and
$\tilde{\Gamma}_{i,j}:W_j\rightarrow W_i$, i.e. $X$ is
decomposable.

\end{enumerate}

\end{proof}

\subsection{Finite type posets}

Recall the the poset $\mathcal N$ has  \emph{finite (linear)
representation type} if there exist only finitely many
indecomposable representation of $\mathcal N$ in
$\rm{Rep}(\mathcal N)$. Result obtained by M.M.Kleiner
\cite{Kleiner1}  gives complete description of the posets of
finite type. In case of primitive posets using connection between
posets and quivers and analogical result for quivers
(see~\cite{GabrielRoiter}, for instance) one can obtain
description of the primitive posets of finite type.

\begin{prop}
   Each primitive poset of finite type has one of the following
   form:

   \begin{itemize}
    \item $(k)$ for all $k\in\mathbb{N}$;
    \item $(k_1,k_2)$ for all $k_1,k_2\in\mathbb{N}$;
    \item $(k,1,1)$, for all $k\in\mathbb{N}$;
    \item $(2,2,1)$;
    \item $(3,2,1)$;
    \item $(4,2,1)$.
\end{itemize}
\end{prop}

\begin{proof}
    \emph{if}

    As it was shown in previous subsection that for every representation of primitive posets
    one can build the representation of corresponding quiver,
    non-isomorphic representations of the poset corresponds to
    non-isomorphic representations of the quiver. The quivers $A_{k+1}$, $A_{k_1+k_2+1}$, $D_{k+3}$,
    $E_6$, $E_7$, $E_8$ are the correspon\-ding quiver for the posets listed in statement. Each of these quivers
    has finitely many of  non-isomorphic
    indecomposable representations (see~\cite{GabrielRoiter}),
    hence the listed posets also has finitely many non-isomorphic indecomposable
    representations.

    \emph{only if}
    One can show that each primitive poset that is not included in
    list contains subposet with corresponding extended Dynkin
    quiver. But this quivers has infinitely many non-isomorphic indecomposable
    representations (see \cite{DlabRingel,DonovanFreishlih} and other for description).
    We list infinite series of non-isomorphic indecomposable
    representations of corresponding poset:

    \begin{itemize}
\item for the poset $(1,1,1,1)$ (the corresponding quiver is $\tilde D_4$):
$$
V=\mathbb{C}^2=\langle e_1, e_2\rangle; V_1=\langle e_1\rangle,
V_2=\langle e_2\rangle, V_3=\langle e_1+e_2\rangle, V_4=\langle
e_1+\lambda e_2\rangle, \lambda\neq 0,1;
$$

\item for the poset $(2,2,2)$ (the corresponding quiver is $\tilde E_6$):
\begin{gather*}
V=\mathbb{C}^3=\langle e_1, e_2, e_3\rangle;\\
V_1=\langle e_1\rangle, V_2=\langle e_1, e_2\rangle;
V_3=\langle e_3\rangle, V_4=\langle e_2, e_3\rangle;\\
V_5=\langle e_1+e_2+e_3\rangle, V_6=\langle e_1+e_2+e_3, \lambda
e_1+e_2\rangle, \lambda\neq 0,1;
\end{gather*}

\item for the poset $(3,3,2)$ (the corresponding quiver is $\tilde E_7$):
\begin{gather*}
V=\mathbb{C}^4=\langle e_1, e_2, e_3, e_4\rangle;\\
V_1=\langle e_1\rangle, V_2=\langle e_1,
e_2\rangle, V_3=\langle e_1, e_2, e_3\rangle;\\
V_4=\langle e_4\rangle, V_5=\langle e_3, e_4\rangle,
V_6=\langle e_2, e_3, e_4\rangle;\\
V_7=\langle e_1+e_2+e_3, \lambda e_1+e_3+e_4\rangle, \lambda\neq
0,1;
\end{gather*}
\item for the poset $(5,2,1)$ (the corresponding quiver is $\tilde E_7$):
\begin{gather*}
V=\mathbb{C}^6=\langle e_1, e_2, e_3, e_4, e_5, e_6\rangle;\\
V_1=\langle e_1\rangle, V_2=\langle e_1, e_2\rangle,V_3=\langle e_1,
e_2, e_3\rangle, V_4=\langle e_1, e_2, e_3, e_4\rangle,
V_5=\langle e_1, e_2, e_3, e_4, e_5\rangle;\\
V_6=\langle e_5, e_6\rangle, V_7=\langle e_3, e_4, e_5, e_6
\rangle;\\
 V_8=\langle e_1+e_3+e_4+e_5, \lambda e_1+e_3+e_5+e_6,
e_2+e_4 \rangle, \lambda\neq 0,1;
\end{gather*}
here vectors $e_i$ form basis in corresponding linear spaces.
\end{itemize}
 Hence the theorem is proved.
\end{proof}

\subsection{Coxeter functors for representations of posets in linear
spaces} \label{linearCoxeter}

Coxeter functors for representations of quivers were studied
in~\cite{BersteinGelfandPonomarev} for proving Gabriel's theorem.

Similar Coxeter functors $F^+, F^-$ acting in the category
$\rm{Rep}(\mathcal N)$ of linear representations of posets was
constructed in~\cite{Drozd}. We are not going to give their full
construction here, but we remind some basic necessary facts.

There are functors $\sigma, \rho$ ($\sigma^2=\rm{id},
\rho^2=\rm{id}$) which act between categories $\rm Rep(\mathcal N)$
and $\rm Rep(\mathcal N^*)$ (where $\mathcal N^*$ is dual to poset
$\mathcal N$). These functors change dimensions of representations
by the following formulas:
\begin{multline*}
(d_0; d^{(1)}_1,\ldots,d^{(1)}_{k_1};
\dots;d^{(m)}_1,\ldots,d^{(m)}_{k_m})
\stackrel{\rho}{\longmapsto}(\sum_{j=1}^md^{(j)}_{k_j}-d_0;
d^{(1)}_{k_1}, \\
d^{(1)}_{k_1}-d^{(1)}_1,\dots,d^{(1)}_{k_1}-d^{(1)}_{k_1-1};\dots;
d^{(m)}_{k_1},\ldots,d^{(m)}_{k_m}-d^{(m)}_{k_m-1});
\end{multline*}
\begin{multline*}
(d_0; d^{(1)}_1,\ldots,d^{(1)}_{k_1};\ldots;d^{(m)}_1,
\dots,d^{(m)}_{k_m})\stackrel{\sigma}{\longmapsto}(d_0;
d_0-d^{(1)}_1,\ldots,d_0-d^{(1)}_{k_1};\ldots; \\
d_0-d^{(m)}_{1},\ldots, d_0-d^{(m)}_{k_m}).
\end{multline*}


By definition $F^+=\rho\sigma,~F^-=\sigma\rho$, hence their action
on generalized dimensions could be written as follows:
\begin{gather}
(d_0; d^{(1)}_1,\ldots,d^{(1)}_{k_1};
\dots;d^{(m)}_1,\ldots,d^{(m)}_{k_m})\stackrel{F^+}{\longmapsto}(\sum_{j=1}^md^{(j)}_{k_j}-d_0;
\sum_{j=2}^md^{(j)}_{k_j}-d_0,\notag\\
\sum_{j=2}^md^{(j)}_{k_j}-d_0+d^{(1)}_1,\ldots,
\sum_{j=2}^md_{k_j}-d_0+d^{(1)}_{k_1-1}; \ldots;
\sum_{j=1}^{m-1}d^{(j)}_{k_j}-d_0+d^{(m)}_{k_m-1});\label{FplusDim}
\end{gather}
\begin{gather}
(d_0; d^{(1)}_1,\ldots,d^{(1)}_{k_1};
\dots;d^{(m)}_1,\ldots,d^{(m)}_{k_m})\stackrel{F^-}{\longmapsto}((m-1)d_0-\sum_{j=1}^md^{(j)}_{k_j};
d^{(1)}_2-d^{(1)}_1,\notag\\
d^{(1)}_3-d^{(1)}_1, \ldots,
d^{(1)}_{k_1-1}-d^{(1)}_1,d_0-d^{(1)}_1;\ldots;
d^{(m)}_{k_l-1}-d^{(m)}_1,d_0-d^{(m)}_1).\label{FminusDim}
\end{gather}

Importance of these functors is due to the following
fact~(see~\cite{Drozd}):

\emph{ Any non-degenerate representation of finite type poset
$\mathcal
N$ is the following: $(F^+)^k\pi$, where $\pi$ is degenerate.
}

\section{Representations of posets in Hilbert spaces}

In this section we will consider the unitary representation of the
posets.  We will show how these representations connected with
$*$-representation of certain algebras that generated by
projections and linear relations. The representations of these
algebras associated with primitive posets could be studied using
the Coxeter functors, which will allow us study unitary
representations of the posets, using similar technique.

\subsection{Unitary representations of posets}
Denote by $\rm{Rep}(\mathcal N, H)$ the sub-category in
$\rm{Rep}(\mathcal N)$. Its set of objects consists of
finite-dimensional Hilbert spaces and two objects $\pi$ and
$\tilde{\pi}$ are equivalent in $\rm{Rep}(\mathcal N)$ if there
exists morphisms between them which is unitary operator
$U:H\rightarrow\tilde H$ such that $U(H_i)= \tilde H_i$ (unitary
equivalent). Representation $\pi\in\rm{Rep}(\mathcal N, H)$ called
\emph{irreducible} iff $C^*$-algebra generated by set of
orthogonal projections $\{P_i\}$ on the subspaces $\{H_i\}$ is
irreducible.
Let remark that indecomposability of representation $\pi$ in
$\rm{Rep}(\mathcal N)$ implies irreducibility of $C^*(\{P_i,
i\in\mathcal N\})$ but converse not
true~(see~\cite{EnomotoWatatani} for details).

The problem of classifying all irreducible objects in category
$\rm{Rep}(\mathcal N, H)$ becomes much harder 
even for primitive poset $\mathcal N=(1,2)$ it is
impossible to classify in reasonable way all irreducible
representations: indeed this lead us to classify up to unitary
equivalence three subspace in Hilbert space, two of which are
orthogonal, but it is well-known~\cite{KruglyaSamoilenko1} that
such task is $*$-wild. Hence it is natural to consider some
additional relation.

Let us consider those objects $\pi \in \rm{Rep}(\mathcal N, H)$,
$\pi=(H;H_1,\ldots,H_2)$, which satisfy the following linear
relation:
\begin{equation}
    \alpha_1P_1+\ldots+\alpha_nP_n=\gamma I, \label{linear_rel}
\end{equation}
where $\alpha_i,\gamma$ are some positive real numbers. These
objects form a category and we will also will denote it by
$\rm{Rep}(\mathcal N, H)$ (this will not make problem, course in
the sequel we will consider only those representations in Hilbert
space of the posets, which satisfy~\eqref{linear_rel}). If to fix
the weight $\chi=(\alpha_1,\ldots,\alpha_n;\gamma) \in \mathbb
R^{n+1}_+$ and consider only representations
satisfying~\eqref{linear_rel}) with this fixed weight one obtain
another category $\rm{Rep}_{\chi}(\mathcal N, H)$ which is
subcategory in $\rm{Rep}(\mathcal N, H)$.

The surprising result of~\cite{KruglyakNazarovaRoiter} is that in
category $\rm{Rep}(\mathcal N)$ each indecomposable representation
equivalent to some object in $\rm{Rep}(\mathcal N, H)$ is brick.

\subsection{Certain $*$-algebras and their representation in the category of Hilbert spaces}

Let be given the weight $\chi=(\alpha_1,\ldots,\alpha_n;\gamma) \in
\mathbb R^{n+1}_+$, consider the following $*$-algebra $\mathcal
A_{\mathcal N,\chi}$, which is generated by $n$ projections and
corresponds to poset $\mathcal N$ of order $n$:
\begin{multline} \label{algebra}
A_{\mathcal N, \chi}:=\mathbb{C}\langle p_1,p_2,\dots,p_n |
p_i=p^2_i=p^*_i, p_i p_j=p_jp_i=p_i, i\prec j, \\
\alpha_1p_1+\ldots+\alpha_np_n=\gamma e\rangle.
\end{multline}
It's easy to see that every $*$-representation $\pi$ of
$*$-algebra $\mathcal{A}_{\mathcal N,\chi}$ gives us
representation of poset $\mathcal N$ in category
$\rm{Rep}_{\chi}(\mathcal N, H)$ ($H_i=\pi(p_i)$), and vice versa
every representation of poset $\mathcal N$ in category
$\rm{Rep}_{\chi}(\mathcal N, H)$  generate the $*$-representation
of $\mathcal{A}_{\mathcal N,\chi}$. To put it another way,
categories $\rm{Rep}_{\chi}(\mathcal N, H)$  and
$\rm{Rep}(\mathcal{A}_{\mathcal N,\chi})$ are equivalent.

For primitive poset $\mathcal N=(k_1,\ldots,k_m)$ weights, for
convenience, will be denoted by $\chi=(\alpha^{(1)}_1,\ldots,
\alpha^{(1)}_{k_1};\ldots;\alpha^{(m)}_{1},\ldots,\alpha^{(m)}_{k_m};\gamma)$
and correspondingly
\begin{multline}
\mathcal{A}_{\mathcal N,\chi}=\langle p^{(1)}_1,\ldots,p^{(1)}_{k_1},\ldots,p^{(m)}_1,\ldots,p^{(m)}_{k_m} | p^{(j)2}_i=p^{(j)*}_i=p^{(j)}_i, \\
p^{(j)}_ip^{(j)}_k=p^{(j)}_kp^{(j)}_i=p^{(j)}_i,
\sum_{j=1}^m\sum_{i=1}^{k_l}\alpha^{(j)}_ip^{(j)}_i=\gamma e
\rangle.
\end{multline}
Under some condition there exists isomorphisms between these
algebras and $*$-algebras $\mathcal{A}_{\Gamma,\hat{\chi}}$ which
are associated  with star-shaped graphs
$\Gamma=T_{k_1+1,\dots,k_m+1}$ and character $\hat
\chi=(\beta^{(1)}_1,\ldots,
\beta^{(1)}_{k_1};\ldots;\beta^{(m)}_{1},\ldots,\beta^{(m)}_{k_m};\gamma)$.
$*$-Algebras $\mathcal{A}_{\Gamma,\hat{\chi}}$ defined  in the
following way:
\begin{multline}
\mathcal{A}_{\Gamma,\hat{\chi}}=\langle
q^{(1)}_1,\ldots,q^{(1)}_{k_1},\ldots,q^{(m)}_1,\ldots,q^{(m)}_{k_m}
| q^{(j)*}_i=q^{(j)}_i,
q^{(j)}_iq^{(j)}_k=q^{(j)}_kq^{(j)}_i=\delta_{ik}q^{(j)}_i, \\
\sum_{j=1}^l\sum_{i=1}^{k_l}\beta^{(j)}_iq^{(j)}_i=\gamma e \rangle.
\end{multline}

\begin{prop}

 Let $\mathcal N$ be primitive poset and $\chi=(\alpha^{(1)}_1,\ldots,
\alpha^{(1)}_{k_1};\ldots;\alpha^{(m)}_{1},\ldots,\alpha^{(m)}_{k_m};\gamma)$
and $\hat \chi=(\beta^{(1)}_1,\ldots,
\beta^{(1)}_{k_1};\ldots;\beta^{(m)}_{1},\ldots,\beta^{(m)}_{k_m};\gamma)$
be two given weights such that for all $j=1,\ldots,m$,
$i=1,\ldots,k_j$ the following conditions hold:
$$
\beta^{(j)}_i=\sum_{s=i}^{k_j}\alpha^{(j)}_s,
$$
then there exist isomorphism $\phi$ between $*$-algebras
$\mathcal{A}_{\mathcal N,\chi}$ and
$\mathcal{A}_{\Gamma,\hat{\chi}}$.
\end{prop}

\begin{proof}
    One can easily proof that map which is given on generators of algebras
    in the following way:
$$
p^{(j)}_i\stackrel{\phi}{\longmapsto}\sum_{s=1}^iq^{(j)}_s,\quad
j=1,\dots, l, i=1,\dots, k_j,
$$
establishes the isomorphism between $\mathcal{A}_{\mathcal N,\chi}$
and $\mathcal{A}_{\Gamma,\hat{\chi}}$.
\end{proof}

The isomorphism $\phi$ can be extended to an equivalence functor
between categories of representations of the algebras. Let $\pi$ be
representation of $*$-algebra $\mathcal{A}_{\Gamma,\hat{\chi}}$ then
$\phi\circ\pi$~-- representation of $\mathcal{A}_{\mathcal N,\chi}$.
Furthermore if $\pi_1$ and $\pi_2$ are equivalent objects in
$\rm{Rep}(\mathcal{A}_{\Gamma,\hat{\chi}})$ then $\phi\circ\pi_1$
and $\phi\circ\pi_2$ are equivalent objects in
$\rm{Rep}(\mathcal{A}_{\mathcal N,\chi})$.

Let as recall some definitions from~\cite{KPS05}:
\begin{defn}
An irreducible finite dimensional $*$-representation $\pi$ of the
algebra $\mathcal{A}_{\Gamma,\hat{\chi}}$ such that
\begin{gather*}
\pi(q^{(j)}_i)\neq 0, j=1,\dots, m, i=1, \dots, k_j) \mbox{ and}\\
\sum_{i=1}^{k_j} \pi(q^{(j)}_i)\neq I, j=1, \dots, m
\end{gather*}
will be called {\it non-degenerate}. By
$\overline{\rm{Rep}}(\mathcal{A}_{\Gamma,\hat{\chi}})$ we will
denote full subcategory of non-degenerate representations in the
category $\rm{Rep}(\mathcal{A}_{\Gamma,\hat{\chi}})$.
\end{defn}

Objects that corespondents to  subcategory
$\overline{\rm{Rep}}(\mathcal{A}_{\Gamma,\hat{\chi}})$ in
$\rm{Rep}\mathcal{A}_{\mathcal N,\chi})$ will be called
non-degenerated as well. Non-degenerate representations of
$*$-algebra $\mathcal{A}_{\mathcal N,\chi}$  corresponds to
non-degenerate representations of poset $\mathcal N$.

If one have degenerate representation of poset $\mathcal N$ in
Hilbert space $H$, that is one of the following conditions are hold:
\begin{itemize}
 \item There exist $i$ such that $H_i=0$ i.e. $P_i=0$;
 \item There exist $i$ such that $H_i=H_{i-1}$ i.e. $P_i=P_{i-1}$;
 \item There exist $i$ such that $H_i=H$ i.e. $P_i=I$
\end{itemize}

In this case one can obtain non-degenerate representation of some
subposet of $\mathcal N$ as follows. If for some $i$, $P_i=0$ then
one automatically have representation of poset $\mathcal
N'=\mathcal N\setminus\{i\}$. If in representation $\pi$ one have
$P_i=P_{i-1}$ than one can consider $\pi$ as representation of
poset $\mathcal N'$ where $\pi(i-1)=H_i=H_{i-1}$ and
$\sum_{j\in\mathcal
N\setminus\{i-1,i\}}\alpha_jP_j+(\alpha_{i-1}+\alpha_i)P_{i-1}=\gamma
I$. And finally, if for some $i$ $\pi(i)=H$ one can consider
representation $\pi$ as representation of $\mathcal N'$ for which
holds $\sum_{j\in\mathcal N'}\alpha_jP_j=(\gamma-\alpha_i) I$.

\subsection{Coxeter functors for representations of $*$-algebras
$\mathcal{A}_{\mathcal N,\chi}$ in case of primitive poset
$\mathcal N$}

 In work \cite{KruglyakRoiter} there
were developed Coxeter functors for representations of quivers in
Hilbert spaces which prescribe so-called 'orthoscalarity'
condition. It was proved that any irreducible representation of
quiver which underlying graph is Dynkin diagram can be obtained
from simplest ones by using this Coxeter functors.

These functors allows to construct some functors ${\Phi}^+$ and
${\Phi}^-$, called by Coxeter too, which act on representations of
algebras $\mathcal{A}_{\Gamma,\hat{\chi}}$. Recall that
$*$-algebras $\mathcal{A}_{\Gamma,\hat{\chi}}$ and their
involutive representations are studied in many recent works
(see~\cite{KPS05,OstrovskyiSamoilenko} and other).
 Using the isomorphism $\phi$ between $*$-algebras
$\mathcal{A}_{\Gamma,\hat{\chi}}$ and $\mathcal{A}_{\mathcal
N,\chi}$ we obtain two new functors which we will be also denoted
by ${\Phi}^+$ and ${\Phi}^-$ and called Coxeter functors. These
functors act from category $\rm Rep(\mathcal{A}_{\mathcal
N,\chi})$ to $\rm Rep(\mathcal{A}_{\mathcal N,\chi_{\Phi^+}})$ and
$\rm Rep(\mathcal{A}_{\mathcal N,\chi_{\Phi^-}})$ respectively,
where
\begin{multline*}
\chi_{\Phi^+}=(\gamma-\sum_{i=1}^{k_1}\alpha^{(1)}_i,
\alpha^{(1)}_1, \dots, \alpha^{(1)}_{k_1-1};\dots;\notag\\
\gamma-\sum_{i=1}^{k_m}\alpha^{(m)}_i, \alpha^{(m)}_1, \dots,
\alpha^{(m)}_{k_m-1};(m-1)\gamma-\sum_{j=1}^m\alpha^{(j)}_{k_j});
\end{multline*}
\begin{multline*}
\chi_{\Phi^-}=(\alpha^{(1)}_2,
\alpha^{(1)}_3,\dots,\alpha^{(1)}_{k_1},
\sum_{j=2}^m\sum_{i=1}^{k_j}\alpha^{(j)}_i-\gamma; \dots;\notag\\
\alpha^{(k_m)}_2, \dots,
\sum_{j=1}^{m-1}\sum_{i=1}^{k_j}\alpha^{(j)}_i-\gamma;
\sum_{j=1}^m\sum_{i=1}^{k_j}\alpha^{(j)}_i-\gamma).
\end{multline*}

In other words, this functors can be considered as functors from
$\rm Rep_{\chi}(\mathcal N,H)$ to $\rm
Rep_{\chi_{\Phi^{\pm}}}(\mathcal N,H)$. Dimensions of
representations are changed by formulas similar
to~\eqref{FminusDim} and~\eqref{FplusDim} respectively.

\section{Unitarization of linear representations of primitive posets}

This section is devoted to the interconnection between linear and
unitary representation of the primitive posets. Actually for
primitive posets of finite type we will obtain the complete list
of possible weights that appropriate to given indecomposable
linear representation.

\subsection{Unitarization}

Let be given the representation $\pi \in \rm{Rep}(\mathcal N)$,
$\pi=(V;V_1,\ldots,V_n)$ of the poset $\mathcal N$. One say that
$\pi$ can be unitarized if there exists appropriate choice of
hermitian structure $\langle\cdot,\cdot\rangle_{\mathbb C}$ in $V$,
that corresponding projection $P_i:V\rightarrow V_i$ satisfy the
following relation:
$$
    \alpha_1P_1+\ldots+\alpha_nP_n=\gamma I,
$$
for some weight $\chi=(\alpha_1,\ldots,\alpha_n;\gamma) \in \mathbb
R^{n+1}_+$, and correspondingly we say that $\pi$ can be unitarized
with weight $\chi$ if this weight is fixed.

Recall that the similar notion of \emph{unitarization} was
provided in article \cite{KruglyakRoiter} for the unitarization of
representations of given quiver $Q$. It fact that work studied the
question when it is possible to define hermitian structure of each
space $X_i, i\in Q_0$, in such way that in each vertex $i$ the
following condition satisfies:
$$
    \sum_{j\longrightarrow i} X_{j\longrightarrow i} X^*_{i\longrightarrow
    j}+\sum_{i\longrightarrow j} X^*_{i\longrightarrow i} X_{j\longrightarrow
    j}=\alpha_i I_{X_i},
$$
where $X^*_{i\longrightarrow j}$ denotes the adjoint map to
$X_{j\longrightarrow i}$ with respect to hermitian structure in
$X_i$ and $X_j$. One of the results of paper~\cite{KruglyakRoiter}
is that if $Q$ is a Dynkin quiver then every representation could
be unitarizable, and if $Q$ is an extended Dynkin quiver then
there are representations that cannot be unitarized.

One can obtain analogical fact for representations of primitive
poset of finite type using Coxeter functos. Noticed that it can be
obtained using results from~\cite{KruglyakRoiter}.

\begin{prop}
Every indecomposable linear representation of primitive poset of
finite type can be unitarized with some weight.
\end{prop}

\begin{proof}
To start with, remark that it is obvious that any one dimensional
representation with dimension vector $d_\pi=(1;d_1,\ldots,d_n)$ of
any poset can be unitarized with weight
$\chi=(\alpha_1,\ldots,\alpha_n;\gamma)$ satisfied \emph{trace
}condition $\sum_{s\in\mathcal N}d_s\alpha_s=\gamma$. All
indecomposable representations of primitive posets
$\mathcal{N}=(k), k\in\mathbb{N}$ and $\mathcal N=(k_1,k_2)$,
$k_1,k_2\in\mathbb{N}$ are one-dimensional hence they could be
unitarized.

Let $\pi$ be some representation of primitive poset $\mathcal N$
of finite type. There are two possibilities: $\pi$ is degenerate
or $\pi$ is non-degenerate. In first case the representation $\pi$
can be considered as one of some subposet $\mathcal N'$ of
$\mathcal N$ than by induction it unitarize and corresponding
unitary representation of $\mathcal N'$ can be restricted to
unitary representation of $\mathcal N$ which is unitarization
of~$\pi$.

In the later case representation $\pi$ by theorem
from~\cite{Drozd} which was mentioned in
section~\ref{linearCoxeter}  is as follows $\pi=(F^+)^k\pi'$ that
is $(F^-)k\pi=\pi'$ where $\pi'$ is degenerate. As degenerate
representation $\pi'$ unitarize with some weight $\chi'$. Using
functor $(\Phi^+)^k$ to corresponded unitary representation we
obtain unitary representation with some weight $\chi$ equivalent
to $\pi$. That is $\pi$ unitarize with weight $\chi$. As the
result we obtain statement of the theorem and even more:  this
gives the algorithm which allows to describe all possible weights
with which given representation of finite type primitive poset can
be unitarized.
\end{proof}

Next theorem gives complete list of all possible weights for
rep\-re\-senta\-tions of primitive posets of finite type. In other
words it describes the set of weights $\chi$ for every primitive
poset $\mathcal N$ such that $*$-algebra $\mathcal{A}_{\mathcal
N,\chi}$ has indecomposable representation in fixed dimention $D$.
Analogical result for $*$-algebras $\mathcal{A}_{\Gamma,\hat{\chi}}$
were obtained in~\cite{KPS05} and following theorem can be obtained
using result of~\cite{KPS05} but we did it independently. The list
of weights are organized in the following way: for each
representation of primitive poset $\mathcal N=(k_1,\ldots,k_n)$
(which for us is given by generalized dimensional
$d=(d_1^{(1)},\ldots,d_{k_1}^{(1)};\ldots;$
$d_1^{(n)},\ldots,d_{k_n}^{(n)};d_0)$, since this gives a
representation up to isomorphism) we state the condition on weights
$\chi=(\alpha_1,\ldots,\alpha_{k_1};\beta_1,\ldots;\gamma)$, under
which it is possible to unitarize the linear representation.

\begin{thm} For primitive poset $\mathcal N$ its linear representation with
dimension $D$ unitarize with every weight $\chi$ for which
conditions $C$ are satisfied:

\vspace{1cm}

 1. Poset $\mathcal N=(1;1;1)$
\begin{picture}(120,0)
      \put(20,0){\circle{4}}
      \put(40,0){\circle{4}}
      \put(60,0) {\circle{4}}
\end{picture}

\vspace{0.5cm}

\begin{longtable}{p{3.2cm} | p{6cm}}

Dimensions $D$ & Conditions $C$ \\
$(0;0;0;1)$ & $\gamma=0 $ \\ \hline $(0;0;1;1)$ & $\delta=\gamma $ \\
\hline $(0;1;0;1)$ & $\beta=\gamma $ \\ \hline $(0;1;1;1)$ &
$\beta+\delta=\gamma $ \\ \hline $(1;0;0;1)$ & $\alpha=\gamma $ \\
\hline $(1;0;1;1)$ & $\alpha+\delta=\gamma $ \\ \hline $(1;1;0;1)$ &
$\alpha+\beta=\gamma $ \\ \hline $(1;1;1;1)$ &
$\alpha+\beta+\delta=\gamma $ \\ \hline $(1;1;1;2)$ &
$\alpha<\gamma, \beta<\gamma, \delta<\gamma,
\alpha+\beta+\delta=2\gamma $ \\ \hline
\end{longtable}

\vspace{0.5cm}

2. Poset $\mathcal N=(2;1;1)$
\begin{picture}(120,20)
      \put(20,0){\circle{4}}
      \put(20,20){\circle{4}}
      \put(20,2){\vector(0,1){16}}
      \put(40,0){\circle{4}}
      \put(60,0) {\circle{4}}
\end{picture}

\vspace{0.5cm}

\begin{longtable}{p{3.2cm} | p{7cm}}

Dimensions & Conditions \\
$(0,0;0;0;1)$ & $ \gamma=0 $ \\ \hline $(0,1;0;0;1)$ & $\alpha_2=\gamma $ \\
\hline $(1,1;0;0;1)$ & $\alpha_1+\alpha_2=\gamma $ \\ \hline
$(0,0;1;0;1)$ & $\beta=\gamma $ \\ \hline $(0,1;1;0;1)$ &
$\alpha_2+\beta=\gamma $ \\ \hline $(1,1;1;0;1)$ &
$\alpha_1+\alpha_2+\beta=\gamma $ \\ \hline $(0,0;0;1;1)$ &
$\delta=\gamma $ \\ \hline $(0,1;0;1;1)$ & $\alpha_2+\delta=\gamma $
\\ \hline $(1,1;0;1;1)$ & $\alpha_1+\alpha_2+\delta=\gamma $ \\
\hline $(0,0;1;1;1)$ & $\beta+\delta=\gamma $ \\ \hline
$(0,1;1;1;1)$ & $\alpha_2+\beta+\delta=\gamma $ \\ \hline
$(1,1;1;1;1)$ & $\alpha_1+\alpha_2+\beta+\delta=\gamma $ \\ \hline
$(0,1;1;1;2)$ & $\alpha_2<\gamma, \beta<\gamma, \delta<\gamma,
\alpha_2+\beta+\delta=2\gamma $ \\ \hline $(1,1;1;1;2)$ &
$\alpha_1+\alpha_2<\gamma, \alpha_1+\alpha_2<\gamma, \beta<\gamma,
\delta<\gamma, \alpha_1+\alpha_2+\beta+\delta=2\gamma $ \\ \hline
$(1,2;1;1;2)$ & $\alpha_1+\alpha_2<\gamma, \alpha_2+\beta<\gamma,
\alpha_2+\delta<\gamma, \alpha_1+2\alpha_2+\beta+\delta=2\gamma $ \\

\end{longtable}

\vspace{0.5cm}

3. Poset $\mathcal N=(2;2;1)$
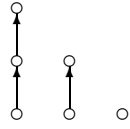
\begin{picture}(120,20)
      \put(20,0){\circle{4}}
      \put(20,20){\circle{4}}
      \put(20,2){\vector(0,1){16}}
      \put(40,0){\circle{4}}
      \put(40,20){\circle{4}}
      \put(40,2){\vector(0,1){16}}
      \put(60,0) {\circle{4}}
\end{picture}

\vspace{0.5cm}

\begin{longtable}{p{3.2cm} | p{7cm}}

Dimensions $D$ & Conditions $C$\\
\hline $(0,0;0,0;0;1)$ & $ \gamma=0$ \\ \hline $(0,1;0,0;0;1)$ &
$\alpha_2=\gamma $
\\ \hline $(0,0;0,0;1;1)$ & $\delta=\gamma $ \\ \hline
$(0,1;0,0;1;1)$ & $\alpha_2+\delta=\gamma $ \\ \hline
$(1,1;0,0;1;1)$ & $\alpha_1+\alpha_2+\delta=\gamma $ \\ \hline
$(0,0;0,1;0;1)$ & $\beta_2=\gamma $ \\ \hline $(0,1;0,1;1;1)$ &
$\alpha_2+\beta_2+\delta=\gamma $ \\ \hline $(0,0;1,1;0;1)$ &
$\beta_1+\beta_2=\gamma $ \\ \hline $(0,1;1,1;0;1)$ &
$\alpha_2+\beta_1+\beta_2=\gamma $ \\ \hline $(0,0;0,1;1;1)$ &
$\beta_2+\delta=\gamma $ \\ \hline $(0,0;1,1;1;1)$ &
$\beta_1+\beta_2+\delta=\gamma $ \\ \hline $(0,1;1,1;1;1)$ &
$\alpha_2+\beta_1+\beta_2+\delta=\gamma $ \\ \hline $(1,1;0,0;0;1)$
& $\alpha_1+\alpha_2=\gamma $ \\ \hline $(0,1;0,1;0;1)$ &
$\alpha_2+\beta_2=\gamma $ \\ \hline $(1,1;0,1;0;1)$ &
$\alpha_1+\alpha_2+\beta_2=\gamma $ \\ \hline $(1,1;0,1;1;1)$ &
$\alpha_1+\alpha_2+\beta_2+\delta=\gamma $ \\ \hline $(1,1;1,1;0;1)$
& $\alpha_1+\alpha_2+\beta_1+\beta_2=\gamma $ \\ \hline
$(1,1;1,1;1;1)$ & $\alpha_1+\alpha_2+\beta_1+\beta_2+\delta=\gamma $
\\ \hline $(0,1;0,1;1;2)$ & $\alpha_2<\gamma, \beta_2<\gamma,
\delta<\gamma, \alpha_2+\beta_2+\delta=2\gamma $ \\ \hline
$(1,1;0,1;1;2)$ & $\alpha_1+\alpha_2<\gamma,
\alpha_1+\alpha_2<\gamma, \beta_2<\gamma, \delta<\gamma,
\alpha_1+\alpha_2+\beta_2+\delta=2\gamma $ \\ \hline $(1,2;0,1;1;2)$
& $\alpha_1+\alpha_2<\gamma, \alpha_2+\beta_2<\gamma,
\alpha_2+\delta<\gamma, \alpha_1+2\alpha_2+\beta_2+\delta=2\gamma $
\\ \hline $(0,1;1,1;1;2)$ & $\alpha_2<\gamma,
\beta_1+\beta_2<\gamma, \beta_1+\beta_2<\gamma, \delta<\gamma,
\alpha_2+\beta_1+\beta_2+\delta=2\gamma $ \\ \hline $(1,1;1,1;1;2)$
& $\alpha_1+\alpha_2<\gamma, \alpha_1+\alpha_2<\gamma,
\beta_1+\beta_2<\gamma, \beta_1+\beta_2<\gamma, \delta<\gamma,
\alpha_1+\alpha_2+\beta_1+\beta_2+\delta=2\gamma $ \\ \hline
$(1,2;1,1;1;2)$ & $\alpha_1+\alpha_2<\gamma,
\alpha_2+\beta_1+\beta_2<\gamma, \alpha_2+\beta_1+\beta_2<\gamma,
\alpha_2+\delta<\gamma,
\alpha_1+2\alpha_2+\beta_1+\beta_2+\delta=2\gamma $ \\ \hline
$(0,1;1,2;1;2)$ & $\alpha_2+\beta_2<\gamma, \beta_1+\beta_2<\gamma,
\beta_2+\delta<\gamma, \alpha_2+\beta_1+2\beta_2+\delta=2\gamma $ \\
\hline $(1,1;1,2;1;2)$ & $\alpha_1+\alpha_2+\beta_2<\gamma,
\alpha_1+\alpha_2+\beta_2<\gamma, \beta_1+\beta_2<\gamma,
\beta_2+\delta<\gamma,
\alpha_1+\alpha_2+\beta_1+2\beta_2+\delta=2\gamma $ \\ \hline
$(1,2;1,2;1;2)$ & $\alpha_1+\alpha_2+\beta_2<\gamma,
\alpha_2+\beta_1+\beta_2<\gamma, \alpha_2+\beta_2+\delta<\gamma,
\alpha_1+2\alpha_2+\beta_1+2\beta_2+\delta=2\gamma $ \\ \hline
$(1,2;1,2;1;3)$ & $\alpha_1+\alpha_2<\gamma,
\alpha_1+\alpha_2+\beta_2+\delta<2\gamma, \beta_1+\beta_2<\gamma,
\alpha_2+\beta_1+\beta_2+\delta<2\gamma, \alpha_2+\beta_2<\gamma,
\alpha_1+2\alpha_2+\beta_1+2\beta_2+\delta=3\gamma $ \\ \hline
$(1,2;1,2;2;3)$ & $\beta_2+\delta<\gamma,
\alpha_2+\beta_1+2\beta_2+\delta<2\gamma, \alpha_2+\delta<\gamma,
\alpha_1+2\alpha_2+\beta_2+\delta<2\gamma,
\alpha_1+\alpha_2+\beta_1+\beta_2+\delta<2\gamma,
\alpha_1+2\alpha_2+\beta_1+2\beta_2+2\delta=3\gamma $ \\

\end{longtable}

\vspace{0.5cm}

4. Poset $\mathcal N=(3;2;1)$
\begin{picture}(120,40)
      \put(20,0){\circle{4}}
      \put(20,20){\circle{4}}
      \put(20,40){\circle{4}}
      \put(20,2){\vector(0,1){16}}
      \put(20,22){\vector(0,1){16}}
      \put(40,0){\circle{4}}
      \put(40,20){\circle{4}}
      \put(40,2){\vector(0,1){16}}
      \put(60,0) {\circle{4}}
\end{picture}

\vspace{0.5cm}

\begin{longtable}{p{3.2cm} | p{7cm}}

Dimensions $D$ & Conditions $C$\\

\hline $(0,0,0;0,0;0;1)$ & $ \gamma=0$ \\ \hline $(0,0,1;0,0;0;1)$ &
$\alpha_3=\gamma $ \\ \hline $(1,1,1;0,0;0;1)$ &
$\alpha_1+\alpha_2+\alpha_3=\gamma $ \\ \hline $(0,0,0;0,1;0;1)$ &
$\beta_2=\gamma $ \\ \hline $(0,0,1;0,1;0;1)$ &
$\alpha_3+\beta_2=\gamma $ \\ \hline $(0,1,1;0,0;0;1)$ &
$\alpha_2+\alpha_3=\gamma $ \\ \hline $(0,1,1;0,1;0;1)$ &
$\alpha_2+\alpha_3+\beta_2=\gamma $ \\ \hline $(1,1,1;0,1;0;1)$ &
$\alpha_1+\alpha_2+\alpha_3+\beta_2=\gamma $ \\ \hline
$(0,0,0;1,1;0;1)$ & $\beta_1+\beta_2=\gamma $ \\ \hline
$(0,0,1;1,1;0;1)$ & $\alpha_3+\beta_1+\beta_2=\gamma $ \\ \hline
$(0,1,1;1,1;0;1)$ & $\alpha_2+\alpha_3+\beta_1+\beta_2=\gamma $ \\
\hline $(1,1,1;1,1;0;1)$ &
$\alpha_1+\alpha_2+\alpha_3+\beta_1+\beta_2=\gamma $ \\ \hline
$(0,0,0;0,0;1;1)$ & $\delta=\gamma $ \\ \hline $(0,0,1;0,0;1;1)$ &
$\alpha_3+\delta=\gamma $ \\ \hline $(0,0,0;0,1;1;1)$ &
$\beta_2+\delta=\gamma $ \\ \hline $(0,0,1;0,1;1;1)$ &
$\alpha_3+\beta_2+\delta=\gamma $ \\ \hline $(0,1,1;0,1;1;1)$ &
$\alpha_2+\alpha_3+\beta_2+\delta=\gamma $ \\ \hline
$(1,1,1;0,1;1;1)$ &
$\alpha_1+\alpha_2+\alpha_3+\beta_2+\delta=\gamma $ \\ \hline
$(0,0,0;1,1;1;1)$ & $\beta_1+\beta_2+\delta=\gamma $ \\ \hline
$(0,0,1;1,1;1;1)$ & $\alpha_3+\beta_1+\beta_2+\delta=\gamma $ \\
\hline $(0,1,1;0,0;1;1)$ & $\alpha_2+\alpha_3+\delta=\gamma $ \\
\hline $(0,1,1;1,1;1;1)$ &
$\alpha_2+\alpha_3+\beta_1+\beta_2+\delta=\gamma $ \\ \hline
$(1,1,1;0,0;1;1)$ & $\alpha_1+\alpha_2+\alpha_3+\delta=\gamma $ \\
\hline $(1,1,1;1,1;1;1)$ &
$\alpha_1+\alpha_2+\alpha_3+\beta_1+\beta_2+\delta=\gamma $ \\
\hline $(0,0,1;0,1;1;2)$ & $\alpha_3<\gamma, \beta_2<\gamma,
\delta<\gamma, \alpha_3+\beta_2+\delta=2\gamma $ \\ \hline
$(0,1,1;0,1;1;2)$ & $\alpha_2+\alpha_3<\gamma,
\alpha_2+\alpha_3<\gamma, \beta_2<\gamma, \delta<\gamma,
\alpha_2+\alpha_3+\beta_2+\delta=2\gamma $ \\ \hline
$(0,1,2;0,1;1;2)$ & $\alpha_2+\alpha_3<\gamma,
\alpha_3+\beta_2<\gamma, \alpha_3+\delta<\gamma,
\alpha_2+2\alpha_3+\beta_2+\delta=2\gamma $ \\ \hline
$(0,0,1;1,1;1;2)$ & $\alpha_3<\gamma, \beta_1+\beta_2<\gamma,
\beta_1+\beta_2<\gamma, \delta<\gamma,
\alpha_3+\beta_1+\beta_2+\delta=2\gamma $ \\ \hline
$(0,1,1;1,1;1;2)$ & $\alpha_2+\alpha_3<\gamma,
\alpha_2+\alpha_3<\gamma, \beta_1+\beta_2<\gamma,
\beta_1+\beta_2<\gamma, \delta<\gamma,
\alpha_2+\alpha_3+\beta_1+\beta_2+\delta=2\gamma $ \\ \hline
$(0,1,2;1,1;1;2)$ & $\alpha_2+\alpha_3<\gamma,
\alpha_3+\beta_1+\beta_2<\gamma, \alpha_3+\beta_1+\beta_2<\gamma,
\alpha_3+\delta<\gamma,
\alpha_2+2\alpha_3+\beta_1+\beta_2+\delta=2\gamma $ \\ \hline
$(0,1,2;1,2;1;2)$ & $\alpha_2+\alpha_3+\beta_2<\gamma,
\alpha_3+\beta_1+\beta_2<\gamma, \alpha_3+\beta_2+\delta<\gamma,
\alpha_2+2\alpha_3+\beta_1+2\beta_2+\delta=2\gamma $ \\ \hline
$(1,1,2;0,1;1;2)$ & $\alpha_1+\alpha_2+\alpha_3<\gamma,
\alpha_1+\alpha_2+\alpha_3<\gamma, \alpha_3+\beta_2<\gamma,
\alpha_3+\delta<\gamma,
\alpha_1+\alpha_2+2\alpha_3+\beta_2+\delta=2\gamma $ \\ \hline
$(1,1,2;1,1;1;2)$ & $\alpha_1+\alpha_2+\alpha_3<\gamma,
\alpha_1+\alpha_2+\alpha_3<\gamma, \alpha_3+\beta_1+\beta_2<\gamma,
\alpha_3+\beta_1+\beta_2<\gamma, \alpha_3+\delta<\gamma,
\alpha_1+\alpha_2+2\alpha_3+\beta_1+\beta_2+\delta=2\gamma $ \\
\hline $(0,0,1;1,2;1;2)$ & $\alpha_3+\beta_2<\gamma,
\beta_1+\beta_2<\gamma, \beta_2+\delta<\gamma,
\alpha_3+\beta_1+2\beta_2+\delta=2\gamma $ \\ \hline
$(0,1,1;1,2;1;2)$ & $\alpha_2+\alpha_3+\beta_2<\gamma,
\alpha_2+\alpha_3+\beta_2<\gamma, \beta_1+\beta_2<\gamma,
\beta_2+\delta<\gamma,
\alpha_2+\alpha_3+\beta_1+2\beta_2+\delta=2\gamma $ \\ \hline
$(1,1,1;0,1;1;2)$ & $\alpha_1+\alpha_2+\alpha_3<\gamma,
\alpha_1+\alpha_2+\alpha_3<\gamma,
\alpha_1+\alpha_2+\alpha_3<\gamma, \beta_2<\gamma, \delta<\gamma,
\alpha_1+\alpha_2+\alpha_3+\beta_2+\delta=2\gamma $ \\ \hline
$(1,1,1;1,1;1;2)$ & $\alpha_1+\alpha_2+\alpha_3<\gamma,
\alpha_1+\alpha_2+\alpha_3<\gamma,
\alpha_1+\alpha_2+\alpha_3<\gamma, \beta_1+\beta_2<\gamma,
\beta_1+\beta_2<\gamma, \delta<\gamma,
\alpha_1+\alpha_2+\alpha_3+\beta_1+\beta_2+\delta=2\gamma $ \\
\hline $(1,1,1;1,2;1;2)$ &
$\alpha_1+\alpha_2+\alpha_3+\beta_2<\gamma,
\alpha_1+\alpha_2+\alpha_3+\beta_2<\gamma,
\alpha_1+\alpha_2+\alpha_3+\beta_2<\gamma, \beta_1+\beta_2<\gamma,
\beta_2+\delta<\gamma,
\alpha_1+\alpha_2+\alpha_3+\beta_1+2\beta_2+\delta=2\gamma $ \\
\hline $(1,1,2;1,2;1;2)$ &
$\alpha_1+\alpha_2+\alpha_3+\beta_2<\gamma,
\alpha_1+\alpha_2+\alpha_3+\beta_2<\gamma,
\alpha_3+\beta_1+\beta_2<\gamma, \alpha_3+\beta_2+\delta<\gamma,
\alpha_1+\alpha_2+2\alpha_3+\beta_1+2\beta_2+\delta=2\gamma $ \\
\hline $(1,2,2;0,1;1;2)$ & $\alpha_1+\alpha_2+\alpha_3<\gamma,
\alpha_2+\alpha_3+\beta_2<\gamma, \alpha_2+\alpha_3+\delta<\gamma,
\alpha_1+2\alpha_2+2\alpha_3+\beta_2+\delta=2\gamma $ \\ \hline
$(1,2,2;1,1;1;2)$ & $\alpha_1+\alpha_2+\alpha_3<\gamma,
\alpha_2+\alpha_3+\beta_1+\beta_2<\gamma,
\alpha_2+\alpha_3+\beta_1+\beta_2<\gamma,
\alpha_2+\alpha_3+\delta<\gamma,
\alpha_1+2\alpha_2+2\alpha_3+\beta_1+\beta_2+\delta=2\gamma $ \\
\hline $(1,2,2;1,2;1;2)$ &
$\alpha_1+\alpha_2+\alpha_3+\beta_2<\gamma,
\alpha_2+\alpha_3+\beta_1+\beta_2<\gamma,
\alpha_2+\alpha_3+\beta_2+\delta<\gamma,
\alpha_1+2\alpha_2+2\alpha_3+\beta_1+2\beta_2+\delta=2\gamma $ \\
\hline $(0,1,2;1,2;1;3)$ & $\alpha_2+\alpha_3<\gamma,
\alpha_2+\alpha_3+\beta_2+\delta<2\gamma, \beta_1+\beta_2<\gamma,
\alpha_3+\beta_1+\beta_2+\delta<2\gamma, \alpha_3+\beta_2<\gamma,
\alpha_2+2\alpha_3+\beta_1+2\beta_2+\delta=3\gamma $ \\ \hline
$(1,1,2;1,2;1;3)$ & $\alpha_1+\alpha_2+\alpha_3<\gamma,
\alpha_1+\alpha_2+\alpha_3<\gamma,
\alpha_1+\alpha_2+\alpha_3+\beta_2+\delta<2\gamma,
\beta_1+\beta_2<\gamma, \alpha_3+\beta_1+\beta_2+\delta<2\gamma,
\alpha_3+\beta_2<\gamma,
\alpha_1+\alpha_2+2\alpha_3+\beta_1+2\beta_2+\delta=3\gamma $ \\
\hline $(1,2,2;1,2;1;3)$ & $\alpha_1+\alpha_2+\alpha_3<\gamma,
\alpha_1+\alpha_2+\alpha_3+\beta_2+\delta<2\gamma,
\alpha_1+\alpha_2+\alpha_3+\beta_2+\delta<2\gamma,
\beta_1+\beta_2<\gamma,
\alpha_2+\alpha_3+\beta_1+\beta_2+\delta<2\gamma,
\alpha_2+\alpha_3+\beta_2<\gamma,
\alpha_1+2\alpha_2+2\alpha_3+\beta_1+2\beta_2+\delta=3\gamma $ \\
\hline $(1,2,3;1,2;1;3)$ & $\alpha_1+\alpha_2+\alpha_3<\gamma,
\alpha_1+\alpha_2+2\alpha_3+\beta_2+\delta<2\gamma,
\alpha_3+\beta_1+\beta_2<\gamma,
\alpha_2+2\alpha_3+\beta_1+\beta_2+\delta<2\gamma,
\alpha_2+\alpha_3+\beta_2<\gamma,
\alpha_1+2\alpha_2+3\alpha_3+\beta_1+2\beta_2+\delta=3\gamma $ \\
\hline $(0,1,2;1,2;2;3)$ & $\beta_2+\delta<\gamma,
\alpha_3+\beta_1+2\beta_2+\delta<2\gamma, \alpha_3+\delta<\gamma,
\alpha_2+2\alpha_3+\beta_2+\delta<2\gamma,
\alpha_2+\alpha_3+\beta_1+\beta_2+\delta<2\gamma,
\alpha_2+2\alpha_3+\beta_1+2\beta_2+2\delta=3\gamma $ \\ \hline
$(1,1,2;1,2;2;3)$ & $\beta_2+\delta<\gamma, \beta_2+\delta<\gamma,
\alpha_3+\beta_1+2\beta_2+\delta<2\gamma, \alpha_3+\delta<\gamma,
\alpha_1+\alpha_2+2\alpha_3+\beta_2+\delta<2\gamma,
\alpha_1+\alpha_2+\alpha_3+\beta_1+\beta_2+\delta<2\gamma,
\alpha_1+\alpha_2+2\alpha_3+\beta_1+2\beta_2+2\delta=3\gamma $ \\
\hline $(1,2,2;1,2;2;3)$ & $\beta_2+\delta<\gamma,
\alpha_2+\alpha_3+\beta_1+2\beta_2+\delta<2\gamma,
\alpha_2+\alpha_3+\beta_1+2\beta_2+\delta<2\gamma,
\alpha_2+\alpha_3+\delta<\gamma,
\alpha_1+2\alpha_2+2\alpha_3+\beta_2+\delta<2\gamma,
\alpha_1+\alpha_2+\alpha_3+\beta_1+\beta_2+\delta<2\gamma,
\alpha_1+2\alpha_2+2\alpha_3+\beta_1+2\beta_2+2\delta=3\gamma $ \\
\hline $(1,2,3;1,2;2;3)$ & $\alpha_3+\beta_2+\delta<\gamma,
\alpha_2+2\alpha_3+\beta_1+2\beta_2+\delta<2\gamma,
\alpha_2+\alpha_3+\delta<\gamma,
\alpha_1+2\alpha_2+2\alpha_3+\beta_2+\delta<2\gamma,
\alpha_1+\alpha_2+2\alpha_3+\beta_1+\beta_2+\delta<2\gamma,
\alpha_1+2\alpha_2+3\alpha_3+\beta_1+2\beta_2+2\delta=3\gamma $ \\
\hline $(1,2,3;1,2;2;4)$ & $\alpha_1+\alpha_2+\alpha_3<\gamma,
\alpha_1+\alpha_2+\alpha_3+\beta_2+\delta<2\gamma,
\alpha_1+\alpha_2+2\alpha_3+\beta_1+2\beta_2+\delta<3\gamma,
\alpha_3+\delta<\gamma, \alpha_2+2\alpha_3+\beta_2+\delta<2\gamma,
\alpha_2+\alpha_3+\beta_1+\beta_2+\delta<2\gamma,
\alpha_1+2\alpha_2+3\alpha_3+\beta_1+2\beta_2+2\delta=4\gamma $ \\
\hline $(1,2,3;1,3;2;4)$ & $\beta_2+\delta<\gamma,
\alpha_3+\beta_1+2\beta_2+\delta<2\gamma,
\alpha_2+2\alpha_3+\beta_1+2\beta_2+2\delta<3\gamma,
\alpha_2+\alpha_3+\beta_2<\gamma,
\alpha_1+2\alpha_2+2\alpha_3+\beta_1+2\beta_2+\delta<3\gamma,
\alpha_1+\alpha_2+2\alpha_3+\beta_2+\delta<2\gamma,
\alpha_1+2\alpha_2+3\alpha_3+\beta_1+3\beta_2+2\delta=4\gamma $ \\
\hline $(1,2,3;2,3;2;4)$ & $\alpha_3+\beta_1+\beta_2<\gamma,
\alpha_2+2\alpha_3+\beta_1+\beta_2+\delta<2\gamma,
\alpha_1+2\alpha_2+3\alpha_3+\beta_1+2\beta_2+\delta<3\gamma,
\alpha_1+\alpha_2+\alpha_3+\beta_1+\beta_2+\delta<2\gamma,
\alpha_1+\alpha_2+2\alpha_3+\beta_1+2\beta_2+2\delta<3\gamma,
\alpha_2+\alpha_3+\beta_1+2\beta_2+\delta<2\gamma,
\alpha_1+2\alpha_2+3\alpha_3+2\beta_1+3\beta_2+2\delta=4\gamma $
\\

\end{longtable}

5. Poset $\mathcal N=(4;2;1)$
\begin{picture}(120,60)
      \put(20,0){\circle{4}}
      \put(20,20){\circle{4}}
      \put(20,40){\circle{4}}
      \put(20,60){\circle{4}}
      \put(20,2){\vector(0,1){16}}
      \put(20,22){\vector(0,1){16}}
      \put(20,42){\vector(0,1){16}}
      \put(40,0){\circle{4}}
      \put(40,20){\circle{4}}
      \put(40,2){\vector(0,1){16}}
      \put(60,0) {\circle{4}}
\end{picture}

\vspace{0.5cm}

\begin{longtable}{p{3.2cm} | p{7cm}}

Dimensions $D$ & Conditions $C$\\
\hline $(0,0,0,0;0,0;0;1)$ & $\gamma=0 $ \\ \hline
$(1,1,1,1;0,0;1;1)$ &
$\alpha_1+\alpha_2+\alpha_3+\alpha_4+\delta=\gamma $ \\ \hline
$(0,0,0,0;0,1;0;1)$ & $\beta_2=\gamma $ \\ \hline
$(0,0,0,0;1,1;1;1)$ & $\beta_1+\beta_2+\delta=\gamma $ \\ \hline
$(0,0,0,1;0,0;0;1)$ & $\alpha_4=\gamma $ \\ \hline
$(0,0,0,0;0,0;1;1)$ & $\delta=\gamma $ \\ \hline $(0,0,0,1;0,0;1;1)$
& $\alpha_4+\delta=\gamma $ \\ \hline $(0,0,0,0;0,1;1;1)$ &
$\beta_2+\delta=\gamma $ \\ \hline $(0,0,0,1;0,1;0;1)$ &
$\alpha_4+\beta_2=\gamma $ \\ \hline $(0,0,0,1;0,1;1;1)$ &
$\alpha_4+\beta_2+\delta=\gamma $ \\ \hline $(0,0,0,0;1,1;0;1)$ &
$\beta_1+\beta_2=\gamma $ \\ \hline $(0,0,0,1;1,1;0;1)$ &
$\alpha_4+\beta_1+\beta_2=\gamma $ \\ \hline $(0,0,0,1;1,1;1;1)$ &
$\alpha_4+\beta_1+\beta_2+\delta=\gamma $ \\ \hline
$(0,0,1,1;0,0;0;1)$ & $\alpha_3+\alpha_4=\gamma $ \\ \hline
$(0,0,1,1;0,0;1;1)$ & $\alpha_3+\alpha_4+\delta=\gamma $ \\ \hline
$(0,0,1,1;1,1;1;1)$ &
$\alpha_3+\alpha_4+\beta_1+\beta_2+\delta=\gamma $ \\ \hline
$(0,1,1,1;0,0;1;1)$ & $\alpha_2+\alpha_3+\alpha_4+\delta=\gamma $ \\
\hline $(0,0,1,1;0,1;1;1)$ &
$\alpha_3+\alpha_4+\beta_2+\delta=\gamma $ \\ \hline
$(0,1,1,1;0,1;1;1)$ &
$\alpha_2+\alpha_3+\alpha_4+\beta_2+\delta=\gamma $ \\ \hline
$(0,1,1,1;1,1;1;1)$ &
$\alpha_2+\alpha_3+\alpha_4+\beta_1+\beta_2+\delta=\gamma $ \\
\hline $(1,1,1,1;0,1;1;1)$ &
$\alpha_1+\alpha_2+\alpha_3+\alpha_4+\beta_2+\delta=\gamma $ \\
\hline $(0,0,1,1;1,1;0;1)$ &
$\alpha_3+\alpha_4+\beta_1+\beta_2=\gamma $ \\ \hline
$(0,1,1,1;0,0;0;1)$ & $\alpha_2+\alpha_3+\alpha_4=\gamma $ \\ \hline
$(0,0,1,1;0,1;0;1)$ & $\alpha_3+\alpha_4+\beta_2=\gamma $ \\ \hline
$(0,1,1,1;0,1;0;1)$ & $\alpha_2+\alpha_3+\alpha_4+\beta_2=\gamma $
\\ \hline $(0,1,1,1;1,1;0;1)$ &
$\alpha_2+\alpha_3+\alpha_4+\beta_1+\beta_2=\gamma $ \\ \hline
$(1,1,1,1;0,0;0;1)$ & $\alpha_1+\alpha_2+\alpha_3+\alpha_4=\gamma $
\\ \hline $(1,1,1,1;0,1;0;1)$ &
$\alpha_1+\alpha_2+\alpha_3+\alpha_4+\beta_2=\gamma $ \\ \hline
$(1,1,1,1;1,1;0;1)$ &
$\alpha_1+\alpha_2+\alpha_3+\alpha_4+\beta_1+\beta_2=\gamma $ \\
\hline $(1,1,1,1;1,1;1;1)$ &
$\alpha_1+\alpha_2+\alpha_3+\alpha_4+\beta_1+\beta_2+\delta=\gamma $
\\ \hline $(0,0,0,1;0,1;1;2)$ & $\alpha_4<\gamma, \beta_2<\gamma,
\delta<\gamma, \alpha_4+\beta_2+\delta=2\gamma $ \\ \hline
$(0,0,1,1;0,1;1;2)$ & $\alpha_3+\alpha_4<\gamma,
\alpha_3+\alpha_4<\gamma, \beta_2<\gamma, \delta<\gamma,
\alpha_3+\alpha_4+\beta_2+\delta=2\gamma $ \\ \hline
$(1,1,1,1;0,1;1;2)$ & $\alpha_1+\alpha_2+\alpha_3+\alpha_4<\gamma,
\alpha_1+\alpha_2+\alpha_3+\alpha_4<\gamma,
\alpha_1+\alpha_2+\alpha_3+\alpha_4<\gamma,
\alpha_1+\alpha_2+\alpha_3+\alpha_4<\gamma, \beta_2<\gamma,
\delta<\gamma,
\alpha_1+\alpha_2+\alpha_3+\alpha_4+\beta_2+\delta=2\gamma $ \\
\hline $(0,0,1,2;0,1;1;2)$ & $\alpha_3+\alpha_4<\gamma,
\alpha_4+\beta_2<\gamma, \alpha_4+\delta<\gamma,
\alpha_3+2\alpha_4+\beta_2+\delta=2\gamma $ \\ \hline
$(0,1,1,1;0,1;1;2)$ & $\alpha_2+\alpha_3+\alpha_4<\gamma,
\alpha_2+\alpha_3+\alpha_4<\gamma,
\alpha_2+\alpha_3+\alpha_4<\gamma, \beta_2<\gamma, \delta<\gamma,
\alpha_2+\alpha_3+\alpha_4+\beta_2+\delta=2\gamma $ \\ \hline
$(1,1,1,2;0,1;1;2)$ & $\alpha_1+\alpha_2+\alpha_3+\alpha_4<\gamma,
\alpha_1+\alpha_2+\alpha_3+\alpha_4<\gamma,
\alpha_1+\alpha_2+\alpha_3+\alpha_4<\gamma, \alpha_4+\beta_2<\gamma,
\alpha_4+\delta<\gamma,
\alpha_1+\alpha_2+\alpha_3+2\alpha_4+\beta_2+\delta=2\gamma $ \\
\hline $(1,1,2,2;0,1;1;2)$ &
$\alpha_1+\alpha_2+\alpha_3+\alpha_4<\gamma,
\alpha_1+\alpha_2+\alpha_3+\alpha_4<\gamma,
\alpha_3+\alpha_4+\beta_2<\gamma, \alpha_3+\alpha_4+\delta<\gamma,
\alpha_1+\alpha_2+2\alpha_3+2\alpha_4+\beta_2+\delta=2\gamma $ \\
\hline $(1,2,2,2;0,1;1;2)$ &
$\alpha_1+\alpha_2+\alpha_3+\alpha_4<\gamma,
\alpha_2+\alpha_3+\alpha_4+\beta_2<\gamma,
\alpha_2+\alpha_3+\alpha_4+\delta<\gamma,
\alpha_1+2\alpha_2+2\alpha_3+2\alpha_4+\beta_2+\delta=2\gamma $ \\
\hline $(0,0,0,1;1,1;1;2)$ & $\alpha_4<\gamma,
\beta_1+\beta_2<\gamma, \beta_1+\beta_2<\gamma, \delta<\gamma,
\alpha_4+\beta_1+\beta_2+\delta=2\gamma $ \\ \hline
$(0,0,1,1;1,1;1;2)$ & $\alpha_3+\alpha_4<\gamma,
\alpha_3+\alpha_4<\gamma, \beta_1+\beta_2<\gamma,
\beta_1+\beta_2<\gamma, \delta<\gamma,
\alpha_3+\alpha_4+\beta_1+\beta_2+\delta=2\gamma $ \\ \hline
$(0,0,1,2;1,1;1;2)$ & $\alpha_3+\alpha_4<\gamma,
\alpha_4+\beta_1+\beta_2<\gamma, \alpha_4+\beta_1+\beta_2<\gamma,
\alpha_4+\delta<\gamma,
\alpha_3+2\alpha_4+\beta_1+\beta_2+\delta=2\gamma $ \\ \hline
$(0,1,1,1;1,1;1;2)$ & $\alpha_2+\alpha_3+\alpha_4<\gamma,
\alpha_2+\alpha_3+\alpha_4<\gamma,
\alpha_2+\alpha_3+\alpha_4<\gamma, \beta_1+\beta_2<\gamma,
\beta_1+\beta_2<\gamma, \delta<\gamma,
\alpha_2+\alpha_3+\alpha_4+\beta_1+\beta_2+\delta=2\gamma $ \\
\hline $(1,1,1,1;1,1;1;2)$ &
$\alpha_1+\alpha_2+\alpha_3+\alpha_4<\gamma,
\alpha_1+\alpha_2+\alpha_3+\alpha_4<\gamma,
\alpha_1+\alpha_2+\alpha_3+\alpha_4<\gamma,
\alpha_1+\alpha_2+\alpha_3+\alpha_4<\gamma, \beta_1+\beta_2<\gamma,
\beta_1+\beta_2<\gamma, \delta<\gamma,
\alpha_1+\alpha_2+\alpha_3+\alpha_4+\beta_1+\beta_2+\delta=2\gamma $
\\ \hline $(1,1,1,2;1,1;1;2)$ &
$\alpha_1+\alpha_2+\alpha_3+\alpha_4<\gamma,
\alpha_1+\alpha_2+\alpha_3+\alpha_4<\gamma,
\alpha_1+\alpha_2+\alpha_3+\alpha_4<\gamma,
\alpha_4+\beta_1+\beta_2<\gamma, \alpha_4+\beta_1+\beta_2<\gamma,
\alpha_4+\delta<\gamma,
\alpha_1+\alpha_2+\alpha_3+2\alpha_4+\beta_1+\beta_2+\delta=2\gamma
$ \\ \hline $(0,1,2,2;0,1;1;2)$ &
$\alpha_2+\alpha_3+\alpha_4<\gamma,
\alpha_3+\alpha_4+\beta_2<\gamma, \alpha_3+\alpha_4+\delta<\gamma,
\alpha_2+2\alpha_3+2\alpha_4+\beta_2+\delta=2\gamma $ \\ \hline
$(0,1,2,2;1,1;1;2)$ & $\alpha_2+\alpha_3+\alpha_4<\gamma,
\alpha_3+\alpha_4+\beta_1+\beta_2<\gamma,
\alpha_3+\alpha_4+\beta_1+\beta_2<\gamma,
\alpha_3+\alpha_4+\delta<\gamma,
\alpha_2+2\alpha_3+2\alpha_4+\beta_1+\beta_2+\delta=2\gamma $ \\
\hline $(1,1,2,2;1,1;1;2)$ &
$\alpha_1+\alpha_2+\alpha_3+\alpha_4<\gamma,
\alpha_1+\alpha_2+\alpha_3+\alpha_4<\gamma,
\alpha_3+\alpha_4+\beta_1+\beta_2<\gamma,
\alpha_3+\alpha_4+\beta_1+\beta_2<\gamma,
\alpha_3+\alpha_4+\delta<\gamma,
\alpha_1+\alpha_2+2\alpha_3+2\alpha_4+\beta_1+\beta_2+\delta=2\gamma
$ \\ \hline $(1,2,2,2;1,1;1;2)$ &
$\alpha_1+\alpha_2+\alpha_3+\alpha_4<\gamma,
\alpha_2+\alpha_3+\alpha_4+\beta_1+\beta_2<\gamma,
\alpha_2+\alpha_3+\alpha_4+\beta_1+\beta_2<\gamma,
\alpha_2+\alpha_3+\alpha_4+\delta<\gamma,
\alpha_1+2\alpha_2+2\alpha_3+2\alpha_4+\beta_1+\beta_2+\delta=2\gamma
$ \\ \hline $(0,0,0,1;1,2;1;2)$ & $\alpha_4+\beta_2<\gamma,
\beta_1+\beta_2<\gamma, \beta_2+\delta<\gamma,
\alpha_4+\beta_1+2\beta_2+\delta=2\gamma $ \\ \hline
$(0,0,1,1;1,2;1;2)$ & $\alpha_3+\alpha_4+\beta_2<\gamma,
\alpha_3+\alpha_4+\beta_2<\gamma, \beta_1+\beta_2<\gamma,
\beta_2+\delta<\gamma,
\alpha_3+\alpha_4+\beta_1+2\beta_2+\delta=2\gamma $ \\ \hline
$(0,0,1,2;1,2;1;2)$ & $\alpha_3+\alpha_4+\beta_2<\gamma,
\alpha_4+\beta_1+\beta_2<\gamma, \alpha_4+\beta_2+\delta<\gamma,
\alpha_3+2\alpha_4+\beta_1+2\beta_2+\delta=2\gamma $ \\ \hline
$(0,1,1,2;0,1;1;2)$ & $\alpha_2+\alpha_3+\alpha_4<\gamma,
\alpha_2+\alpha_3+\alpha_4<\gamma, \alpha_4+\beta_2<\gamma,
\alpha_4+\delta<\gamma,
\alpha_2+\alpha_3+2\alpha_4+\beta_2+\delta=2\gamma $ \\ \hline
$(0,1,1,2;1,1;1;2)$ & $\alpha_2+\alpha_3+\alpha_4<\gamma,
\alpha_2+\alpha_3+\alpha_4<\gamma, \alpha_4+\beta_1+\beta_2<\gamma,
\alpha_4+\beta_1+\beta_2<\gamma, \alpha_4+\delta<\gamma,
\alpha_2+\alpha_3+2\alpha_4+\beta_1+\beta_2+\delta=2\gamma $ \\
\hline $(0,1,1,1;1,2;1;2)$ &
$\alpha_2+\alpha_3+\alpha_4+\beta_2<\gamma,
\alpha_2+\alpha_3+\alpha_4+\beta_2<\gamma,
\alpha_2+\alpha_3+\alpha_4+\beta_2<\gamma, \beta_1+\beta_2<\gamma,
\beta_2+\delta<\gamma,
\alpha_2+\alpha_3+\alpha_4+\beta_1+2\beta_2+\delta=2\gamma $ \\
\hline $(0,1,1,2;1,2;1;2)$ &
$\alpha_2+\alpha_3+\alpha_4+\beta_2<\gamma,
\alpha_2+\alpha_3+\alpha_4+\beta_2<\gamma,
\alpha_4+\beta_1+\beta_2<\gamma, \alpha_4+\beta_2+\delta<\gamma,
\alpha_2+\alpha_3+2\alpha_4+\beta_1+2\beta_2+\delta=2\gamma $ \\
\hline $(1,1,1,2;1,2;1;2)$ &
$\alpha_1+\alpha_2+\alpha_3+\alpha_4+\beta_2<\gamma,
\alpha_1+\alpha_2+\alpha_3+\alpha_4+\beta_2<\gamma,
\alpha_1+\alpha_2+\alpha_3+\alpha_4+\beta_2<\gamma,
\alpha_4+\beta_1+\beta_2<\gamma, \alpha_4+\beta_2+\delta<\gamma,
\alpha_1+\alpha_2+\alpha_3+2\alpha_4+\beta_1+2\beta_2+\delta=2\gamma
$ \\ \hline $(0,1,2,2;1,2;1;2)$ &
$\alpha_2+\alpha_3+\alpha_4+\beta_2<\gamma,
\alpha_3+\alpha_4+\beta_1+\beta_2<\gamma,
\alpha_3+\alpha_4+\beta_2+\delta<\gamma,
\alpha_2+2\alpha_3+2\alpha_4+\beta_1+2\beta_2+\delta=2\gamma $ \\
\hline $(1,1,1,1;1,2;1;2)$ &
$\alpha_1+\alpha_2+\alpha_3+\alpha_4+\beta_2<\gamma,
\alpha_1+\alpha_2+\alpha_3+\alpha_4+\beta_2<\gamma,
\alpha_1+\alpha_2+\alpha_3+\alpha_4+\beta_2<\gamma,
\alpha_1+\alpha_2+\alpha_3+\alpha_4+\beta_2<\gamma,
\beta_1+\beta_2<\gamma, \beta_2+\delta<\gamma,
\alpha_1+\alpha_2+\alpha_3+\alpha_4+\beta_1+2\beta_2+\delta=2\gamma
$ \\ \hline $(1,1,2,2;1,2;1;2)$ &
$\alpha_1+\alpha_2+\alpha_3+\alpha_4+\beta_2<\gamma,
\alpha_1+\alpha_2+\alpha_3+\alpha_4+\beta_2<\gamma,
\alpha_3+\alpha_4+\beta_1+\beta_2<\gamma,
\alpha_3+\alpha_4+\beta_2+\delta<\gamma,
\alpha_1+\alpha_2+2\alpha_3+2\alpha_4+\beta_1+2\beta_2+\delta=2\gamma
$ \\ \hline $(1,2,2,2;1,2;1;2)$ &
$\alpha_1+\alpha_2+\alpha_3+\alpha_4+\beta_2<\gamma,
\alpha_2+\alpha_3+\alpha_4+\beta_1+\beta_2<\gamma,
\alpha_2+\alpha_3+\alpha_4+\beta_2+\delta<\gamma,
\alpha_1+2\alpha_2+2\alpha_3+2\alpha_4+\beta_1+2\beta_2+\delta=2\gamma
$ \\ \hline $(0,0,1,2;1,2;1;3)$ & $\alpha_3+\alpha_4<\gamma,
\alpha_3+\alpha_4+\beta_2+\delta<2\gamma, \beta_1+\beta_2<\gamma,
\alpha_4+\beta_1+\beta_2+\delta<2\gamma, \alpha_4+\beta_2<\gamma,
\alpha_3+2\alpha_4+\beta_1+2\beta_2+\delta=3\gamma $ \\ \hline
$(0,1,1,2;1,2;1;3)$ & $\alpha_2+\alpha_3+\alpha_4<\gamma,
\alpha_2+\alpha_3+\alpha_4<\gamma,
\alpha_2+\alpha_3+\alpha_4+\beta_2+\delta<2\gamma,
\beta_1+\beta_2<\gamma, \alpha_4+\beta_1+\beta_2+\delta<2\gamma,
\alpha_4+\beta_2<\gamma,
\alpha_2+\alpha_3+2\alpha_4+\beta_1+2\beta_2+\delta=3\gamma $ \\
\hline $(0,1,2,2;1,2;1;3)$ & $\alpha_2+\alpha_3+\alpha_4<\gamma,
\alpha_2+\alpha_3+\alpha_4+\beta_2+\delta<2\gamma,
\alpha_2+\alpha_3+\alpha_4+\beta_2+\delta<2\gamma,
\beta_1+\beta_2<\gamma,
\alpha_3+\alpha_4+\beta_1+\beta_2+\delta<2\gamma,
\alpha_3+\alpha_4+\beta_2<\gamma,
\alpha_2+2\alpha_3+2\alpha_4+\beta_1+2\beta_2+\delta=3\gamma $ \\
\hline $(0,0,1,2;1,2;2;3)$ & $\beta_2+\delta<\gamma,
\alpha_4+\beta_1+2\beta_2+\delta<2\gamma, \alpha_4+\delta<\gamma,
\alpha_3+2\alpha_4+\beta_2+\delta<2\gamma,
\alpha_3+\alpha_4+\beta_1+\beta_2+\delta<2\gamma,
\alpha_3+2\alpha_4+\beta_1+2\beta_2+2\delta=3\gamma $ \\ \hline
$(0,1,1,2;1,2;2;3)$ & $\beta_2+\delta<\gamma, \beta_2+\delta<\gamma,
\alpha_4+\beta_1+2\beta_2+\delta<2\gamma, \alpha_4+\delta<\gamma,
\alpha_2+\alpha_3+2\alpha_4+\beta_2+\delta<2\gamma,
\alpha_2+\alpha_3+\alpha_4+\beta_1+\beta_2+\delta<2\gamma,
\alpha_2+\alpha_3+2\alpha_4+\beta_1+2\beta_2+2\delta=3\gamma $ \\
\hline $(0,1,2,2;1,2;2;3)$ & $\beta_2+\delta<\gamma,
\alpha_3+\alpha_4+\beta_1+2\beta_2+\delta<2\gamma,
\alpha_3+\alpha_4+\beta_1+2\beta_2+\delta<2\gamma,
\alpha_3+\alpha_4+\delta<\gamma,
\alpha_2+2\alpha_3+2\alpha_4+\beta_2+\delta<2\gamma,
\alpha_2+\alpha_3+\alpha_4+\beta_1+\beta_2+\delta<2\gamma,
\alpha_2+2\alpha_3+2\alpha_4+\beta_1+2\beta_2+2\delta=3\gamma $ \\
\hline $(0,1,2,3;1,2;1;3)$ & $\alpha_2+\alpha_3+\alpha_4<\gamma,
\alpha_2+\alpha_3+2\alpha_4+\beta_2+\delta<2\gamma,
\alpha_4+\beta_1+\beta_2<\gamma,
\alpha_3+2\alpha_4+\beta_1+\beta_2+\delta<2\gamma,
\alpha_3+\alpha_4+\beta_2<\gamma,
\alpha_2+2\alpha_3+3\alpha_4+\beta_1+2\beta_2+\delta=3\gamma $ \\
\hline $(0,1,2,3;1,2;2;3)$ & $\alpha_4+\beta_2+\delta<\gamma,
\alpha_3+2\alpha_4+\beta_1+2\beta_2+\delta<2\gamma,
\alpha_3+\alpha_4+\delta<\gamma,
\alpha_2+2\alpha_3+2\alpha_4+\beta_2+\delta<2\gamma,
\alpha_2+\alpha_3+2\alpha_4+\beta_1+\beta_2+\delta<2\gamma,
\alpha_2+2\alpha_3+3\alpha_4+\beta_1+2\beta_2+2\delta=3\gamma $ \\
\hline $(1,1,1,2;1,2;1;3)$ &
$\alpha_1+\alpha_2+\alpha_3+\alpha_4<\gamma,
\alpha_1+\alpha_2+\alpha_3+\alpha_4<\gamma,
\alpha_1+\alpha_2+\alpha_3+\alpha_4<\gamma,
\alpha_1+\alpha_2+\alpha_3+\alpha_4+\beta_2+\delta<2\gamma,
\beta_1+\beta_2<\gamma, \alpha_4+\beta_1+\beta_2+\delta<2\gamma,
\alpha_4+\beta_2<\gamma,
\alpha_1+\alpha_2+\alpha_3+2\alpha_4+\beta_1+2\beta_2+\delta=3\gamma
$ \\ \hline $(1,1,2,2;1,2;1;3)$ &
$\alpha_1+\alpha_2+\alpha_3+\alpha_4<\gamma,
\alpha_1+\alpha_2+\alpha_3+\alpha_4<\gamma,
\alpha_1+\alpha_2+\alpha_3+\alpha_4+\beta_2+\delta<2\gamma,
\alpha_1+\alpha_2+\alpha_3+\alpha_4+\beta_2+\delta<2\gamma,
\beta_1+\beta_2<\gamma,
\alpha_3+\alpha_4+\beta_1+\beta_2+\delta<2\gamma,
\alpha_3+\alpha_4+\beta_2<\gamma,
\alpha_1+\alpha_2+2\alpha_3+2\alpha_4+\beta_1+2\beta_2+\delta=3\gamma
$ \\ \hline $(1,1,2,3;1,2;1;3)$ &
$\alpha_1+\alpha_2+\alpha_3+\alpha_4<\gamma,
\alpha_1+\alpha_2+\alpha_3+\alpha_4<\gamma,
\alpha_1+\alpha_2+\alpha_3+2\alpha_4+\beta_2+\delta<2\gamma,
\alpha_4+\beta_1+\beta_2<\gamma,
\alpha_3+2\alpha_4+\beta_1+\beta_2+\delta<2\gamma,
\alpha_3+\alpha_4+\beta_2<\gamma,
\alpha_1+\alpha_2+2\alpha_3+3\alpha_4+\beta_1+2\beta_2+\delta=3\gamma
$ \\ \hline $(1,1,1,2;1,2;2;3)$ & $\beta_2+\delta<\gamma,
\beta_2+\delta<\gamma, \beta_2+\delta<\gamma,
\alpha_4+\beta_1+2\beta_2+\delta<2\gamma, \alpha_4+\delta<\gamma,
\alpha_1+\alpha_2+\alpha_3+2\alpha_4+\beta_2+\delta<2\gamma,
\alpha_1+\alpha_2+\alpha_3+\alpha_4+\beta_1+\beta_2+\delta<2\gamma,
\alpha_1+\alpha_2+\alpha_3+2\alpha_4+\beta_1+2\beta_2+2\delta=3\gamma
$ \\ \hline $(1,1,2,2;1,2;2;3)$ & $\beta_2+\delta<\gamma,
\beta_2+\delta<\gamma,
\alpha_3+\alpha_4+\beta_1+2\beta_2+\delta<2\gamma,
\alpha_3+\alpha_4+\beta_1+2\beta_2+\delta<2\gamma,
\alpha_3+\alpha_4+\delta<\gamma,
\alpha_1+\alpha_2+2\alpha_3+2\alpha_4+\beta_2+\delta<2\gamma,
\alpha_1+\alpha_2+\alpha_3+\alpha_4+\beta_1+\beta_2+\delta<2\gamma,
\alpha_1+\alpha_2+2\alpha_3+2\alpha_4+\beta_1+2\beta_2+2\delta=3\gamma
$ \\ \hline $(1,1,2,3;1,2;2;3)$ & $\alpha_4+\beta_2+\delta<\gamma,
\alpha_4+\beta_2+\delta<\gamma,
\alpha_3+2\alpha_4+\beta_1+2\beta_2+\delta<2\gamma,
\alpha_3+\alpha_4+\delta<\gamma,
\alpha_1+\alpha_2+2\alpha_3+2\alpha_4+\beta_2+\delta<2\gamma,
\alpha_1+\alpha_2+\alpha_3+2\alpha_4+\beta_1+\beta_2+\delta<2\gamma,
\alpha_1+\alpha_2+2\alpha_3+3\alpha_4+\beta_1+2\beta_2+2\delta=3\gamma
$ \\ \hline $(1,2,2,2;1,2;1;3)$ &
$\alpha_1+\alpha_2+\alpha_3+\alpha_4<\gamma,
\alpha_1+\alpha_2+\alpha_3+\alpha_4+\beta_2+\delta<2\gamma,
\alpha_1+\alpha_2+\alpha_3+\alpha_4+\beta_2+\delta<2\gamma,
\alpha_1+\alpha_2+\alpha_3+\alpha_4+\beta_2+\delta<2\gamma,
\beta_1+\beta_2<\gamma,
\alpha_2+\alpha_3+\alpha_4+\beta_1+\beta_2+\delta<2\gamma,
\alpha_2+\alpha_3+\alpha_4+\beta_2<\gamma,
\alpha_1+2\alpha_2+2\alpha_3+2\alpha_4+\beta_1+2\beta_2+\delta=3\gamma
$ \\ \hline $(1,2,2,3;1,2;1;3)$ &
$\alpha_1+\alpha_2+\alpha_3+\alpha_4<\gamma,
\alpha_1+\alpha_2+\alpha_3+2\alpha_4+\beta_2+\delta<2\gamma,
\alpha_1+\alpha_2+\alpha_3+2\alpha_4+\beta_2+\delta<2\gamma,
\alpha_4+\beta_1+\beta_2<\gamma,
\alpha_2+\alpha_3+2\alpha_4+\beta_1+\beta_2+\delta<2\gamma,
\alpha_2+\alpha_3+\alpha_4+\beta_2<\gamma,
\alpha_1+2\alpha_2+2\alpha_3+3\alpha_4+\beta_1+2\beta_2+\delta=3\gamma
$ \\ \hline $(1,2,2,2;1,2;2;3)$ & $\beta_2+\delta<\gamma,
\alpha_2+\alpha_3+\alpha_4+\beta_1+2\beta_2+\delta<2\gamma,
\alpha_2+\alpha_3+\alpha_4+\beta_1+2\beta_2+\delta<2\gamma,
\alpha_2+\alpha_3+\alpha_4+\beta_1+2\beta_2+\delta<2\gamma,
\alpha_2+\alpha_3+\alpha_4+\delta<\gamma,
\alpha_1+2\alpha_2+2\alpha_3+2\alpha_4+\beta_2+\delta<2\gamma,
\alpha_1+\alpha_2+\alpha_3+\alpha_4+\beta_1+\beta_2+\delta<2\gamma,
\alpha_1+2\alpha_2+2\alpha_3+2\alpha_4+\beta_1+2\beta_2+2\delta=3\gamma
$ \\ \hline $(1,2,2,3;1,2;2;3)$ & $\alpha_4+\beta_2+\delta<\gamma,
\alpha_2+\alpha_3+2\alpha_4+\beta_1+2\beta_2+\delta<2\gamma,
\alpha_2+\alpha_3+2\alpha_4+\beta_1+2\beta_2+\delta<2\gamma,
\alpha_2+\alpha_3+\alpha_4+\delta<\gamma,
\alpha_1+2\alpha_2+2\alpha_3+2\alpha_4+\beta_2+\delta<2\gamma,
\alpha_1+\alpha_2+\alpha_3+2\alpha_4+\beta_1+\beta_2+\delta<2\gamma,
\alpha_1+2\alpha_2+2\alpha_3+3\alpha_4+\beta_1+2\beta_2+2\delta=3\gamma
$ \\ \hline $(1,2,3,3;1,2;1;3)$ &
$\alpha_1+\alpha_2+\alpha_3+\alpha_4<\gamma,
\alpha_1+\alpha_2+2\alpha_3+2\alpha_4+\beta_2+\delta<2\gamma,
\alpha_3+\alpha_4+\beta_1+\beta_2<\gamma,
\alpha_2+2\alpha_3+2\alpha_4+\beta_1+\beta_2+\delta<2\gamma,
\alpha_2+\alpha_3+\alpha_4+\beta_2<\gamma,
\alpha_1+2\alpha_2+3\alpha_3+3\alpha_4+\beta_1+2\beta_2+\delta=3\gamma
$ \\ \hline $(1,2,3,3;1,2;2;3)$ &
$\alpha_3+\alpha_4+\beta_2+\delta<\gamma,
\alpha_2+2\alpha_3+2\alpha_4+\beta_1+2\beta_2+\delta<2\gamma,
\alpha_2+\alpha_3+\alpha_4+\delta<\gamma,
\alpha_1+2\alpha_2+2\alpha_3+2\alpha_4+\beta_2+\delta<2\gamma,
\alpha_1+\alpha_2+2\alpha_3+2\alpha_4+\beta_1+\beta_2+\delta<2\gamma,
\alpha_1+2\alpha_2+3\alpha_3+3\alpha_4+\beta_1+2\beta_2+2\delta=3\gamma
$ \\ \hline $(0,1,2,3;1,2;2;4)$ &
$\alpha_2+\alpha_3+\alpha_4<\gamma,
\alpha_2+\alpha_3+\alpha_4+\beta_2+\delta<2\gamma,
\alpha_2+\alpha_3+2\alpha_4+\beta_1+2\beta_2+\delta<3\gamma,
\alpha_4+\delta<\gamma, \alpha_3+2\alpha_4+\beta_2+\delta<2\gamma,
\alpha_3+\alpha_4+\beta_1+\beta_2+\delta<2\gamma,
\alpha_2+2\alpha_3+3\alpha_4+\beta_1+2\beta_2+2\delta=4\gamma $ \\
\hline $(1,1,2,3;1,2;2;4)$ &
$\alpha_1+\alpha_2+\alpha_3+\alpha_4<\gamma,
\alpha_1+\alpha_2+\alpha_3+\alpha_4<\gamma,
\alpha_1+\alpha_2+\alpha_3+\alpha_4+\beta_2+\delta<2\gamma,
\alpha_1+\alpha_2+\alpha_3+2\alpha_4+\beta_1+2\beta_2+\delta<3\gamma,
\alpha_4+\delta<\gamma, \alpha_3+2\alpha_4+\beta_2+\delta<2\gamma,
\alpha_3+\alpha_4+\beta_1+\beta_2+\delta<2\gamma,
\alpha_1+\alpha_2+2\alpha_3+3\alpha_4+\beta_1+2\beta_2+2\delta=4\gamma
$ \\ \hline $(1,2,2,3;1,2;2;4)$ &
$\alpha_1+\alpha_2+\alpha_3+\alpha_4<\gamma,
\alpha_1+\alpha_2+\alpha_3+\alpha_4+\beta_2+\delta<2\gamma,
\alpha_1+\alpha_2+\alpha_3+\alpha_4+\beta_2+\delta<2\gamma,
\alpha_1+\alpha_2+\alpha_3+2\alpha_4+\beta_1+2\beta_2+\delta<3\gamma,
\alpha_4+\delta<\gamma,
\alpha_2+\alpha_3+2\alpha_4+\beta_2+\delta<2\gamma,
\alpha_2+\alpha_3+\alpha_4+\beta_1+\beta_2+\delta<2\gamma,
\alpha_1+2\alpha_2+2\alpha_3+3\alpha_4+\beta_1+2\beta_2+2\delta=4\gamma
$ \\ \hline $(1,2,3,3;1,2;2;4)$ &
$\alpha_1+\alpha_2+\alpha_3+\alpha_4<\gamma,
\alpha_1+\alpha_2+\alpha_3+\alpha_4+\beta_2+\delta<2\gamma,
\alpha_1+\alpha_2+2\alpha_3+2\alpha_4+\beta_1+2\beta_2+\delta<3\gamma,
\alpha_1+\alpha_2+2\alpha_3+2\alpha_4+\beta_1+2\beta_2+\delta<3\gamma,
\alpha_3+\alpha_4+\delta<\gamma,
\alpha_2+2\alpha_3+2\alpha_4+\beta_2+\delta<2\gamma,
\alpha_2+\alpha_3+\alpha_4+\beta_1+\beta_2+\delta<2\gamma,
\alpha_1+2\alpha_2+3\alpha_3+3\alpha_4+\beta_1+2\beta_2+2\delta=4\gamma
$ \\ \hline $(1,2,3,4;1,2;2;4)$ &
$\alpha_1+\alpha_2+\alpha_3+\alpha_4<\gamma,
\alpha_1+\alpha_2+\alpha_3+2\alpha_4+\beta_2+\delta<2\gamma,
\alpha_1+\alpha_2+2\alpha_3+3\alpha_4+\beta_1+2\beta_2+\delta<3\gamma,
\alpha_3+\alpha_4+\delta<\gamma,
\alpha_2+2\alpha_3+2\alpha_4+\beta_2+\delta<2\gamma,
\alpha_2+\alpha_3+2\alpha_4+\beta_1+\beta_2+\delta<2\gamma,
\alpha_1+2\alpha_2+3\alpha_3+4\alpha_4+\beta_1+2\beta_2+2\delta=4\gamma
$ \\ \hline $(0,1,2,3;1,3;2;4)$ & $\beta_2+\delta<\gamma,
\alpha_4+\beta_1+2\beta_2+\delta<2\gamma,
\alpha_3+2\alpha_4+\beta_1+2\beta_2+2\delta<3\gamma,
\alpha_3+\alpha_4+\beta_2<\gamma,
\alpha_2+2\alpha_3+2\alpha_4+\beta_1+2\beta_2+\delta<3\gamma,
\alpha_2+\alpha_3+2\alpha_4+\beta_2+\delta<2\gamma,
\alpha_2+2\alpha_3+3\alpha_4+\beta_1+3\beta_2+2\delta=4\gamma $ \\
\hline $(1,1,2,3;1,3;2;4)$ & $\beta_2+\delta<\gamma,
\beta_2+\delta<\gamma, \alpha_4+\beta_1+2\beta_2+\delta<2\gamma,
\alpha_3+2\alpha_4+\beta_1+2\beta_2+2\delta<3\gamma,
\alpha_3+\alpha_4+\beta_2<\gamma,
\alpha_1+\alpha_2+2\alpha_3+2\alpha_4+\beta_1+2\beta_2+\delta<3\gamma,
\alpha_1+\alpha_2+\alpha_3+2\alpha_4+\beta_2+\delta<2\gamma,
\alpha_1+\alpha_2+2\alpha_3+3\alpha_4+\beta_1+3\beta_2+2\delta=4\gamma
$ \\ \hline $(1,2,2,3;1,3;2;4)$ & $\beta_2+\delta<\gamma,
\alpha_4+\beta_1+2\beta_2+\delta<2\gamma,
\alpha_4+\beta_1+2\beta_2+\delta<2\gamma,
\alpha_2+\alpha_3+2\alpha_4+\beta_1+2\beta_2+2\delta<3\gamma,
\alpha_2+\alpha_3+\alpha_4+\beta_2<\gamma,
\alpha_1+2\alpha_2+2\alpha_3+2\alpha_4+\beta_1+2\beta_2+\delta<3\gamma,
\alpha_1+\alpha_2+\alpha_3+2\alpha_4+\beta_2+\delta<2\gamma,
\alpha_1+2\alpha_2+2\alpha_3+3\alpha_4+\beta_1+3\beta_2+2\delta=4\gamma
$ \\ \hline $(1,2,3,3;1,3;2;4)$ & $\beta_2+\delta<\gamma,
\alpha_3+\alpha_4+\beta_1+2\beta_2+\delta<2\gamma,
\alpha_2+2\alpha_3+2\alpha_4+\beta_1+2\beta_2+2\delta<3\gamma,
\alpha_2+2\alpha_3+2\alpha_4+\beta_1+2\beta_2+2\delta<3\gamma,
\alpha_2+\alpha_3+\alpha_4+\beta_2<\gamma,
\alpha_1+2\alpha_2+2\alpha_3+2\alpha_4+\beta_1+2\beta_2+\delta<3\gamma,
\alpha_1+\alpha_2+2\alpha_3+2\alpha_4+\beta_2+\delta<2\gamma,
\alpha_1+2\alpha_2+3\alpha_3+3\alpha_4+\beta_1+3\beta_2+2\delta=4\gamma
$ \\ \hline $(1,2,3,4;1,3;2;4)$ & $\alpha_4+\beta_2+\delta<\gamma,
\alpha_3+2\alpha_4+\beta_1+2\beta_2+\delta<2\gamma,
\alpha_2+2\alpha_3+3\alpha_4+\beta_1+2\beta_2+2\delta<3\gamma,
\alpha_2+\alpha_3+\alpha_4+\beta_2<\gamma,
\alpha_1+2\alpha_2+2\alpha_3+3\alpha_4+\beta_1+2\beta_2+\delta<3\gamma,
\alpha_1+\alpha_2+2\alpha_3+2\alpha_4+\beta_2+\delta<2\gamma,
\alpha_1+2\alpha_2+3\alpha_3+4\alpha_4+\beta_1+3\beta_2+2\delta=4\gamma
$ \\ \hline $(0,1,2,3;2,3;2;4)$ & $\alpha_4+\beta_1+\beta_2<\gamma,
\alpha_3+2\alpha_4+\beta_1+\beta_2+\delta<2\gamma,
\alpha_2+2\alpha_3+3\alpha_4+\beta_1+2\beta_2+\delta<3\gamma,
\alpha_2+\alpha_3+\alpha_4+\beta_1+\beta_2+\delta<2\gamma,
\alpha_2+\alpha_3+2\alpha_4+\beta_1+2\beta_2+2\delta<3\gamma,
\alpha_3+\alpha_4+\beta_1+2\beta_2+\delta<2\gamma,
\alpha_2+2\alpha_3+3\alpha_4+2\beta_1+3\beta_2+2\delta=4\gamma $ \\
\hline $(1,1,2,3;2,3;2;4)$ & $\alpha_4+\beta_1+\beta_2<\gamma,
\alpha_4+\beta_1+\beta_2<\gamma,
\alpha_3+2\alpha_4+\beta_1+\beta_2+\delta<2\gamma,
\alpha_1+\alpha_2+2\alpha_3+3\alpha_4+\beta_1+2\beta_2+\delta<3\gamma,
\alpha_1+\alpha_2+\alpha_3+\alpha_4+\beta_1+\beta_2+\delta<2\gamma,
\alpha_1+\alpha_2+\alpha_3+2\alpha_4+\beta_1+2\beta_2+2\delta<3\gamma,
\alpha_3+\alpha_4+\beta_1+2\beta_2+\delta<2\gamma,
\alpha_1+\alpha_2+2\alpha_3+3\alpha_4+2\beta_1+3\beta_2+2\delta=4\gamma
$ \\ \hline $(1,2,2,3;2,3;2;4)$ & $\alpha_4+\beta_1+\beta_2<\gamma,
\alpha_2+\alpha_3+2\alpha_4+\beta_1+\beta_2+\delta<2\gamma,
\alpha_2+\alpha_3+2\alpha_4+\beta_1+\beta_2+\delta<2\gamma,
\alpha_1+2\alpha_2+2\alpha_3+3\alpha_4+\beta_1+2\beta_2+\delta<3\gamma,
\alpha_1+\alpha_2+\alpha_3+\alpha_4+\beta_1+\beta_2+\delta<2\gamma,
\alpha_1+\alpha_2+\alpha_3+2\alpha_4+\beta_1+2\beta_2+2\delta<3\gamma,
\alpha_2+\alpha_3+\alpha_4+\beta_1+2\beta_2+\delta<2\gamma,
\alpha_1+2\alpha_2+2\alpha_3+3\alpha_4+2\beta_1+3\beta_2+2\delta=4\gamma
$ \\ \hline $(1,2,3,3;2,3;2;4)$ &
$\alpha_3+\alpha_4+\beta_1+\beta_2<\gamma,
\alpha_2+2\alpha_3+2\alpha_4+\beta_1+\beta_2+\delta<2\gamma,
\alpha_1+2\alpha_2+3\alpha_3+3\alpha_4+\beta_1+2\beta_2+\delta<3\gamma,
\alpha_1+2\alpha_2+3\alpha_3+3\alpha_4+\beta_1+2\beta_2+\delta<3\gamma,
\alpha_1+\alpha_2+\alpha_3+\alpha_4+\beta_1+\beta_2+\delta<2\gamma,
\alpha_1+\alpha_2+2\alpha_3+2\alpha_4+\beta_1+2\beta_2+2\delta<3\gamma,
\alpha_2+\alpha_3+\alpha_4+\beta_1+2\beta_2+\delta<2\gamma,
\alpha_1+2\alpha_2+3\alpha_3+3\alpha_4+2\beta_1+3\beta_2+2\delta=4\gamma
$ \\ \hline $(1,2,3,4;2,3;2;4)$ &
$\alpha_3+\alpha_4+\beta_1+\beta_2<\gamma,
\alpha_2+2\alpha_3+2\alpha_4+\beta_1+\beta_2+\delta<2\gamma,
\alpha_1+2\alpha_2+3\alpha_3+3\alpha_4+\beta_1+2\beta_2+\delta<3\gamma,
\alpha_1+\alpha_2+\alpha_3+2\alpha_4+\beta_1+\beta_2+\delta<2\gamma,
\alpha_1+\alpha_2+2\alpha_3+3\alpha_4+\beta_1+2\beta_2+2\delta<3\gamma,
\alpha_2+\alpha_3+2\alpha_4+\beta_1+2\beta_2+\delta<2\gamma,
\alpha_1+2\alpha_2+3\alpha_3+4\alpha_4+2\beta_1+3\beta_2+2\delta=4\gamma
$ \\ \hline $(1,2,3,4;1,3;2;5)$ &
$\alpha_1+\alpha_2+\alpha_3+\alpha_4<\gamma,
\alpha_1+\alpha_2+\alpha_3+\alpha_4+\beta_2+\delta<2\gamma,
\alpha_1+\alpha_2+\alpha_3+2\alpha_4+\beta_1+2\beta_2+\delta<3\gamma,
\alpha_1+\alpha_2+2\alpha_3+3\alpha_4+\beta_1+2\beta_2+2\delta<4\gamma,
\alpha_3+\alpha_4+\beta_2<\gamma,
\alpha_2+2\alpha_3+2\alpha_4+\beta_1+2\beta_2+\delta<3\gamma,
\alpha_2+\alpha_3+2\alpha_4+\beta_2+\delta<2\gamma,
\alpha_1+2\alpha_2+3\alpha_3+4\alpha_4+\beta_1+3\beta_2+2\delta=5\gamma
$ \\ \hline $(1,2,3,4;2,3;2;5)$ &
$\alpha_1+\alpha_2+\alpha_3+\alpha_4<\gamma,
\alpha_4+\beta_1+\beta_2<\gamma,
\alpha_3+2\alpha_4+\beta_1+\beta_2+\delta<2\gamma,
\alpha_2+2\alpha_3+3\alpha_4+\beta_1+2\beta_2+\delta<3\gamma,
\alpha_2+\alpha_3+\alpha_4+\beta_1+\beta_2+\delta<2\gamma,
\alpha_1+2\alpha_2+2\alpha_3+3\alpha_4+\beta_1+2\beta_2+2\delta<4\gamma,
\alpha_1+\alpha_2+2\alpha_3+2\alpha_4+\beta_1+2\beta_2+\delta<3\gamma,
\alpha_1+2\alpha_2+3\alpha_3+4\alpha_4+2\beta_1+3\beta_2+2\delta=5\gamma
$ \\ \hline $(1,2,3,4;2,4;2;5)$ & $\alpha_4+\beta_1+\beta_2<\gamma,
\alpha_2+\alpha_3+\alpha_4+\beta_2<\gamma,
\alpha_1+2\alpha_2+2\alpha_3+2\alpha_4+\beta_1+2\beta_2+\delta<3\gamma,
\alpha_1+2\alpha_2+2\alpha_3+3\alpha_4+\beta_1+3\beta_2+2\delta<4\gamma,
\alpha_3+\alpha_4+\beta_1+2\beta_2+\delta<2\gamma,
\alpha_2+2\alpha_3+3\alpha_4+2\beta_1+3\beta_2+2\delta<4\gamma,
\alpha_1+\alpha_2+2\alpha_3+3\alpha_4+\beta_1+2\beta_2+\delta<3\gamma,
\alpha_1+2\alpha_2+3\alpha_3+4\alpha_4+2\beta_1+4\beta_2+2\delta=5\gamma
$ \\ \hline $(1,2,3,4;1,3;3;5)$ & $\beta_2+\delta<\gamma,
\alpha_3+\alpha_4+\delta<\gamma,
\alpha_2+2\alpha_3+2\alpha_4+\beta_2+\delta<2\gamma,
\alpha_1+2\alpha_2+3\alpha_3+3\alpha_4+\beta_1+2\beta_2+2\delta<4\gamma,
\alpha_1+\alpha_2+\alpha_3+2\alpha_4+\beta_2+\delta<2\gamma,
\alpha_1+\alpha_2+2\alpha_3+3\alpha_4+\beta_1+3\beta_2+2\delta<4\gamma,
\alpha_2+\alpha_3+2\alpha_4+\beta_1+2\beta_2+2\delta<3\gamma,
\alpha_1+2\alpha_2+3\alpha_3+4\alpha_4+\beta_1+3\beta_2+3\delta=5\gamma
$ \\ \hline $(1,2,3,4;2,3;3;5)$ & $\alpha_3+\alpha_4+\delta<\gamma,
\alpha_1+\alpha_2+\alpha_3+\alpha_4+\beta_1+\beta_2+\delta<2\gamma,
\alpha_1+\alpha_2+\alpha_3+2\alpha_4+\beta_1+2\beta_2+2\delta<3\gamma,
\alpha_1+\alpha_2+2\alpha_3+3\alpha_4+2\beta_1+3\beta_2+2\delta<4\gamma,
\alpha_2+\alpha_3+2\alpha_4+\beta_1+\beta_2+\delta<2\gamma,
\alpha_1+2\alpha_2+3\alpha_3+4\alpha_4+\beta_1+2\beta_2+2\delta<4\gamma,
\alpha_2+2\alpha_3+2\alpha_4+\beta_1+2\beta_2+2\delta<3\gamma,
\alpha_1+2\alpha_2+3\alpha_3+4\alpha_4+2\beta_1+3\beta_2+3\delta=5\gamma
$ \\ \hline $(1,2,3,4;2,4;3;5)$ & $\alpha_4+\beta_2+\delta<\gamma,
\alpha_3+2\alpha_4+\beta_1+2\beta_2+\delta<2\gamma,
\alpha_2+2\alpha_3+3\alpha_4+\beta_1+2\beta_2+2\delta<3\gamma,
\alpha_1+2\alpha_2+3\alpha_3+4\alpha_4+\beta_1+3\beta_2+2\delta<4\gamma,
\alpha_2+\alpha_3+\alpha_4+\beta_1+2\beta_2+\delta<2\gamma,
\alpha_1+2\alpha_2+2\alpha_3+3\alpha_4+2\beta_1+3\beta_2+2\delta<4\gamma,
\alpha_1+\alpha_2+2\alpha_3+2\alpha_4+\beta_1+2\beta_2+2\delta<3\gamma,
\alpha_1+2\alpha_2+3\alpha_3+4\alpha_4+2\beta_1+4\beta_2+3\delta=5\gamma
$ \\ \hline $(1,2,3,4;2,4;3;6)$ & $\beta_2+\delta<\gamma,
\alpha_4+\beta_1+2\beta_2+\delta<2\gamma,
\alpha_3+2\alpha_4+\beta_1+2\beta_2+2\delta<3\gamma,
\alpha_2+2\alpha_3+3\alpha_4+\beta_1+3\beta_2+2\delta<4\gamma,
\alpha_2+\alpha_3+\alpha_4+\beta_1+\beta_2+\delta<2\gamma,
\alpha_1+2\alpha_2+2\alpha_3+3\alpha_4+\beta_1+2\beta_2+2\delta<4\gamma,
\alpha_1+\alpha_2+2\alpha_3+2\alpha_4+\beta_1+2\beta_2+\delta<3\gamma,
\alpha_1+2\alpha_2+3\alpha_3+4\alpha_4+2\beta_1+4\beta_2+3\delta=6\gamma
$ \\ \hline $(1,2,3,5;2,4;3;6)$ & $\alpha_4+\beta_1+\beta_2<\gamma,
\alpha_3+2\alpha_4+\beta_1+\beta_2+\delta<2\gamma,
\alpha_2+2\alpha_3+3\alpha_4+\beta_1+2\beta_2+\delta<3\gamma,
\alpha_1+2\alpha_2+3\alpha_3+4\alpha_4+2\beta_1+3\beta_2+2\delta<5\gamma,
\alpha_1+\alpha_2+\alpha_3+2\alpha_4+\beta_2+\delta<2\gamma,
\alpha_1+\alpha_2+2\alpha_3+3\alpha_4+\beta_1+3\beta_2+2\delta<4\gamma,
\alpha_2+\alpha_3+2\alpha_4+\beta_1+2\beta_2+2\delta<3\gamma,
\alpha_1+2\alpha_2+3\alpha_3+5\alpha_4+2\beta_1+4\beta_2+3\delta=6\gamma
$ \\ \hline $(1,2,4,5;2,4;3;6)$ & $\alpha_3+\alpha_4+\delta<\gamma,
\alpha_2+2\alpha_3+2\alpha_4+\beta_2+\delta<2\gamma,
\alpha_1+2\alpha_2+3\alpha_3+3\alpha_4+\beta_1+2\beta_2+2\delta<4\gamma,
\alpha_1+2\alpha_2+3\alpha_3+4\alpha_4+\beta_1+3\beta_2+3\delta<5\gamma,
\alpha_3+\alpha_4+\beta_1+2\beta_2+\delta<2\gamma,
\alpha_2+2\alpha_3+3\alpha_4+2\beta_1+3\beta_2+2\delta<4\gamma,
\alpha_1+\alpha_2+2\alpha_3+3\alpha_4+\beta_1+2\beta_2+\delta<3\gamma,
\alpha_1+2\alpha_2+4\alpha_3+5\alpha_4+2\beta_1+4\beta_2+3\delta=6\gamma
$ \\ \hline $(1,3,4,5;2,4;3;6)$ &
$\alpha_2+\alpha_3+\alpha_4+\beta_2<\gamma,
\alpha_1+2\alpha_2+2\alpha_3+2\alpha_4+\beta_1+2\beta_2+\delta<3\gamma,
\alpha_1+2\alpha_2+2\alpha_3+3\alpha_4+\beta_1+3\beta_2+2\delta<4\gamma,
\alpha_1+2\alpha_2+3\alpha_3+4\alpha_4+2\beta_1+4\beta_2+2\delta<5\gamma,
\alpha_2+\alpha_3+2\alpha_4+\beta_1+\beta_2+\delta<2\gamma,
\alpha_1+2\alpha_2+3\alpha_3+4\alpha_4+\beta_1+2\beta_2+2\delta<4\gamma,
\alpha_2+2\alpha_3+2\alpha_4+\beta_1+2\beta_2+2\delta<3\gamma,
\alpha_1+3\alpha_2+4\alpha_3+5\alpha_4+2\beta_1+4\beta_2+3\delta=6\gamma
$ \\ \hline $(2,3,4,5;2,4;3;6)$ &
$\alpha_1+\alpha_2+\alpha_3+\alpha_4+\beta_1+\beta_2+\delta<2\gamma,
\alpha_1+\alpha_2+\alpha_3+2\alpha_4+\beta_1+2\beta_2+2\delta<3\gamma,
\alpha_1+\alpha_2+2\alpha_3+3\alpha_4+2\beta_1+3\beta_2+2\delta<4\gamma,
\alpha_1+2\alpha_2+3\alpha_3+4\alpha_4+2\beta_1+3\beta_2+3\delta<5\gamma,
\alpha_1+\alpha_2+2\alpha_3+2\alpha_4+\beta_2+\delta<2\gamma,
\alpha_1+2\alpha_2+3\alpha_3+3\alpha_4+\beta_1+3\beta_2+2\delta<4\gamma,
\alpha_1+2\alpha_2+2\alpha_3+3\alpha_4+\beta_1+2\beta_2+\delta<3\gamma,
2\alpha_1+3\alpha_2+4\alpha_3+5\alpha_4+2\beta_1+4\beta_2+3\delta=6\gamma
$ \\

\end{longtable}

\end{thm}

\vspace{0.5cm}

 \noindent \textbf{Acknowledgments.} The authors are
deeply appreciate to Prof. Yu.Samoilenko for the statement of
problem and for his constant attention to this work.

\end{document}